\documentclass[twoside,openright]{amsart}

\newtheorem{thm}{Theorem}[section]
\newtheorem{cor}[thm]{Corollary}
\newtheorem{lem}[thm]{Lemma}
\newtheorem{prop}[thm]{Proposition}

\theoremstyle{definition}
\newtheorem{defn}[thm]{Definition}
\newtheorem{rmk}[thm]{Remark}
\numberwithin{equation}{section}

\usepackage{amssymb,amsmath,latexsym,amsfonts,amscd,amsthm,bm}

\usepackage{amsmath}
\usepackage{epsfig,amsthm}
\usepackage{latexsym}
\usepackage{amsfonts}
\usepackage{amssymb}
\usepackage{amscd}
\usepackage{mathrsfs}
\usepackage[all,cmtip]{xy}
\usepackage{enumerate}
\usepackage{color}

\usepackage[colorlinks]{hyperref}
\hypersetup{pdfstartview={FitH},         linkcolor=blue,  citecolor=green }

\begin{document}

\title[Some endoscopic properties ...]{Some endoscopic properties of the essentially tame Jacquet-Langlands correspondence
}

\author{Kam Fai Tam}
\address{Dept. of Math. \& Stat.,
McMaster University,
1280 Main Street West,
Hamilton, Ontario,
Canada L8S 4K1}
\email{geotam@math.mcmaster.ca}

\begin{abstract}
Let $F$ be a a non-Archimedean local field of characteristic 0 and $G$ be an inner form of the general linear group $G^*=\mathrm{GL}_n$ over $F$. We show that the rectifying character appearing in the essentially tame Jacquet-Langlands correspondence of Bushnell and Henniart for $G$ and $G^*$ can be factorized into a product of some special characters, called zeta-data in this paper, in the theory of endoscopy of Langlands and Shelstad. As a consequence, the essentially tame local Langlands correspondence for $G$ can be described using admissible embeddings of L-tori.

\end{abstract}
\maketitle

\tableofcontents

\section{Introduction}

Let $G^*$ be a general linear group over a non-Archimedean local field of characteristic 0, and $G$ be an inner form of $G^*$. In this paper, we refine the results of the rectifying characters in the context of the essentially tame Jacquet-Langlands correspondence \cite{BH-JLC} by proving that each rectifying character admits a factorization into a product of characters called $\zeta$-data, defined similarly to $\chi$-data in \cite{LS}, which are significant in describing the essentially tame local Langlands correspondence for $G$.

 We know from \cite{BH-JLC} that the rectifying characters are quadratic characters that measure the difference between two correspondences for essentially tame supercuspidal representations of $G$ and $G^*$: the representation theoretic one by matching the maximal simple types of the two groups, and the functorial one from the Jacquet-Langlands correspondence. On the representation theoretic side, the maximal simple types of $G$ can be constructed using certain characters of its elliptic maximal tori, while on the functorial side, the Langlands parameters for $G$ can be functorially lifted from the parameters of the same collection of characters.

 Combining these results with our result on rectifying characters, we show that the essentially tame local Langlands correspondence for $G$ can be described completely by admissible embeddings, defined in \cite{LS}, of the L-groups of elliptic maximal tori into the L-group of $G$, generalizing an analogous description proved by the author \cite{thesis} in the split case (when $G=G^*$).

\subsection{Background}

Let $F$ be a non-Archimedean local field, $G^*$ be the group ${\mathrm{GL}}_n$ defined over $F$, and $G$ be an inner form of $G^*$ defined over $F$. The set of $F$-points $G(F)$ of $G$ is therefore isomorphic to ${\mathrm{GL}}_m(D)$ as a group, where $D$ is a central division algebra over $F$ of dimension $d^2$ and $m=n/d$.

Let $\mathcal{A}_m^2(D)$ (resp. $\mathcal{A}_n^2(F)$) be the discrete series of $G(F)$ (resp. $G^*(F)$), i.e., the set of equivalence classes of irreducible admissible representations that are essentially square integrable mod-center. The Jacquet-Langlands correspondence asserts a bijection
\begin{equation*}
  JL:\mathcal{A}^2_n(F)\rightarrow \mathcal{A}^2_m(D)
\end{equation*}
determined by a character relation (see (\ref{JL-character-relation})) between a representation in $\mathcal{A}^2_n(F)$ and its image in $\mathcal{A}^2_m(D)$. The existence of this bijection is known, starting from the case when the characteristic char$(F)$ of $F$ is 0 and $n=2$ \cite{Jac-Langlands-GL(2)}, in which case $G(F)$ is the multiplicative group of the quaternion algebra over $F$. For arbitrary $n$, when $G(F)$ is the multiplicative group of a division algebra, the existence is proved by \cite{rogawski}. The general situations are treated by \cite{DKV} in the characteristic zero case and by \cite{Badu}, \cite{BHL} in the positive characteristic case.

Bushnell and Henniart describe in \cite[(2.1)]{BH-JLC} the image of $JL$ when it is restricted to the subset of representations in $\mathcal{A}^2_m(D)$, each of whose \emph{parametric degree} is equal to $n$. We do not need the full definition of the parametric degree of a representation, so we only refer to \cite[Section 2.7]{BH-JLC} for details. We only need to know that
\begin{itemize}
  \item the parametric degree of a representation in $\mathcal{A}^2_m(D)$ is a positive integer divisor of $n$, and
  \item a representation in $\mathcal{A}^2_m(D)$ is supercuspidal if its parametric degree is $n$; the converse is true in the split case (when $G=G^*$) but not in general.
\end{itemize}

Furthermore, we can describe the image of $JL$ of the representations in the above subset which are \emph{essentially tame}, a notion we will explain in Section \ref{section The correspondences}. More precisely, if we let $ \mathcal{A}^{{\mathrm{et}}}_m(D)$ (resp. $\mathcal{A}_n^{{\mathrm{et}}}(F)$) be the set of essentially tame representations in $ \mathcal{A}_m^2(D)$ (resp. $\mathcal{A}_n^2(F)$) whose parametric degree is $n$, then we can describe completely the \emph{essentially tame Jacquet-Langlands correspondence}:
\begin{equation*}
  JL:\mathcal{A}^{{\mathrm{et}}}_n(F)\rightarrow \mathcal{A}^{{\mathrm{et}}}_m(D),
\end{equation*}
as in \cite[Theorem A]{BH-JLC}.

To explain the theorem and describe $JL$ completely, we require the notion of admissible characters from \cite{Howe}. In Section \ref{section Admissible characters}, we define the set $P_n(F)$ of (equivalence classes of) admissible pairs $(E/F,\xi)\in P_n(F)$, where $E/F$ is a tamely ramified extension of degree $n$ and $\xi$ is a character of $E^\times$ admissible over $F$. This set bijectively parametrizes both $\mathcal{A}^{\mathrm{et}}_n(F)$ and $\mathcal{A}^{\mathrm{et}}_m(D)$ explicitly \cite{BH-JLC}, using the theory of simple types of $G(F)$ developed in \cite{BF-non-ab-Gauss-sum}, \cite{Grab}, \cite{Secherre1}, \cite{Secherre2}, \cite{Secherre3}, \cite{Secherre4}, \cite{Secherre5} which generalizes the corresponding theory in the split case \cite{BK}, \cite{BH-LTL1} and the division algebra case \cite{zink-D}, \cite{Brou-div-alg}.

If we denote by
\begin{equation}\label{intro bijection of P and FPi}
  {}_F\Pi:P_n(F)\rightarrow\mathcal{A}^\mathrm{et}_n(F),\,(E/F,\xi)\mapsto {}_F\Pi_\xi
\end{equation}
 and
 \begin{equation}\label{intro bijection of P and DPi}
   {}_D\Pi:P_n(F)\rightarrow\mathcal{A}^\mathrm{et}_m(D),\,(E/F,\xi)\mapsto {}_D\Pi_\xi
 \end{equation}
 the above bijections, then Bushnell and Henniart proved in \cite{BH-JLC} that the composition
\begin{equation*}
\nu: P_n(F) \xrightarrow{{}_F\Pi} \mathcal{A}^\mathrm{et}_n(F)
\xrightarrow{JL} \mathcal{A}^\mathrm{et}_m(D)
\xrightarrow{{}_D\Pi^{-1}} P_n(F)
\end{equation*}
maps an admissible pair $(E/F,\xi)\in P_n(F)$ to another pair of the form $(E/F,\xi\cdot{}_D\nu_\xi)$, where ${}_D\nu_\xi$ is a tamely ramified character ${}_D\nu_\xi$ of $E^\times$ depending on $\xi$. We borrow the terminology from \cite{BH-ET3} and call the character ${}_D\nu_\xi$ the \emph{rectifier} of $\xi$ for the essentially tame Jacquet-Langlands correspondence.

\subsection{Main results}

The main result of this paper is to relate the rectifier ${}_D\nu_\xi$ with a special set of characters, called $\zeta$-data in this paper, introduced in the theory of endoscopy of Langlands and Shelstad \cite{LS}. The significance of $\zeta$-data will be explained in the next section, together with a brief summary of the previous
results of the author \cite{thesis}.

To describe the main result, we first assume that char$(F)=0$ (see Remark \ref{characteristic} about this assumption). Let $T$ be the torus such that $T(F)=E^\times$. We regard $T$ as a maximal torus embedded in $G$. In contrast to the split case, we have to carefully choose the embedding $T\rightarrow G$ relative to an hereditary $\mathfrak{o}_F$-order in $G(F)$ associated to $\xi$. This will be explained in Section \ref{subsection An embedding condition}. Given this embedding, let $\Phi=\Phi(G,T)$ be the root system. For each root $\lambda\in \Phi$, we denote by $E_\lambda$ the fixed field of the stabilizer of $\lambda$ in the absolute Galois group $\Gamma_F$ of $F$, so that $E_\lambda$ is a field extension of a $\Gamma_F$-conjugate of $E$. We recall from \cite[Corollary 2.5.B]{LS} that $\zeta$-data is a set of characters
$$\{\zeta_{\lambda}\}=\{\zeta_{\lambda}\}_{\lambda\in \Gamma_F\backslash \Phi},$$
$\text{where each }\zeta_\lambda\text{ is a character of }E_\lambda^\times$ satisfying the conditions in \emph{loc. cit.} (and will be recalled in Section \ref{section Admissible embeddings of tori}). Here $\lambda$ ranges over a suitable subset of roots in $\Phi$, denoted by $\Gamma_F\backslash \Phi$ for the moment, representing the $\Gamma_F$-orbits of $\Phi$ and such that $E_\lambda$ is a field extension of $E$ (but not just its conjugate).

The following theorem restates the main result, Theorem \ref{zeta-data factor of BH-rectifier}, in a simpler way.
\begin{thm}\label{intro-main-thm-1} Given a character $\xi$ of $E^\times$ admissible over $F$.
  \begin{enumerate}[(i)]
    \item There exists a set of $\zeta$-data $\{\zeta_{\lambda,\xi}\}_{\lambda\in \Gamma_F\backslash\Phi}$ such that
  \begin{equation*}
  {}_D\nu_\xi=\prod_{\lambda\in \Gamma_F\backslash \Phi}\zeta_{\lambda,\xi}|_{E^\times}.
\end{equation*}
\label{zeta-data factorize}
\item The values of each $\zeta_{\lambda,\xi}$ can be expressed in terms of certain invariants, called t-factors in this paper, of the corresponding component in the complete symmetric decomposition of the finite symplectic modules associated to $\xi$ (see the notations and definitions in Sections \ref{section Complete decomposition of finite symplectic modules} and \ref{section
Invariants of finite symplectic modules}). \label{main-theorem-module}
  \end{enumerate}
\qed\end{thm}
We explain statement (\ref{main-theorem-module}) of the above theorem. The finite symplectic modules appear in the respective constructions of the extended maximal simple types inducing $ {}_F\Pi_\xi$ and ${}_D\Pi_\xi$ in (\ref{intro bijection of P and FPi}) and (\ref{intro bijection of P and DPi}) (see \cite[(2.5.4)]{BH-JLC}, or the summary in Section \ref{section local construction, attached pair}). Each of these modules admits an orthogonal decomposition, called a complete symmetric decomposition in this paper (Proposition \ref{complete decomp orthogonal}), whose components are parametrized by the same set $\Gamma_F\backslash \Phi$ parameterizing the factors in  (\ref{zeta-data factorize}) of the Theorem. The t-factors are, roughly speaking, defined by the symplectic signs attached to these components.

When proving Theorem \ref{intro-main-thm-1}, we pick a choice of characters $\{\zeta_{\lambda,\xi}\}_{\lambda\in \Gamma_F\backslash\Phi}$, where each character $\zeta_{\lambda,\xi}$ has values in terms of the t-factors of the corresponding component, the one indexed by $\lambda$. We then show that these characters constitute a set of $\zeta$-data. Moreover, using the multiplicativity of t-factors, the product of these $\zeta$-data, when restricted to $E^\times$, is equal to the rectifier ${}_D\nu_\xi$, whose values are given in the First and Second Comparison Theorems of \cite{BH-JLC}. Hence our result refines the one in \emph{loc. cit.}.

While the finite symplectic modules and their decompositions are also studied by the author in the split case \cite{thesis}, there are extra conditions on the components of these modules in the general case. These conditions come from the extra ramifications of the related compact subgroups in constructing the extended maximal simple types. The degrees of these ramifications depend on the residue degree $f(E/F)$ and, with other conditions similar to those in the split case, determine whether each component is trivial or not. This new phenomenon will be fully studied in Section \ref{chapter Finite symplectic modules}. In particular, when $E/F$ is totally ramified and $\xi$ is fixed, the finite symplectic modules are isomorphic to each other for all inner forms of $G^*$, a fact already known in \cite[Proposition 5.6]{BH-JLC}.

\subsection{Relation with previous results}\label{section-relation-previous}

The significance of the factorization of ${}_D\nu_\xi$ in Theorem \ref{intro-main-thm-1}(\ref{zeta-data factorize})  comes from \cite{thesis}, which proves an analogous factorization of the rectifier ${}_F\mu_\xi$ for the essentially tame local Langlands correspondence \cite{BH-ET1}.

We first recall from \emph{loc. cit.} that the rectifier ${}_F\mu_\xi$ measures the difference between the ``na{\"i}ve correspondence" and the essentially tame Langlands correspondence for $G^*$; more precisely, the Langlands parameter of ${}_F\Pi_\xi$ defined in (\ref{intro bijection of P and FPi}) is the induced representation
\begin{equation}\label{paramater-as-induced-repres}
  {\mathrm{Ind}}_{\mathcal{W}_E}^{\mathcal{W}_F}(\xi\cdot{}_F\mu_\xi)
\end{equation}
of the Weil group $\mathcal{W}_F$ of $F$, where $\xi\cdot{}_F\mu_\xi$  is regarded as a character of $\mathcal{W}_E$ by class field theory \cite{Tate-NTB}. In \cite[Theorem 1.1]{thesis}, the author proved that the rectifier ${}_F\mu_\xi$ admits a factorization
\begin{equation*}\label{chi-data factorize}
  {}_F\mu_\xi=\prod_{\lambda\in \Gamma_F\backslash
\Phi}\chi_{\lambda,\xi}|_{E^\times},
\end{equation*}
where $\{\chi_{\lambda,\xi}\}_{\lambda\in\Gamma_F\backslash \Phi}$ is a set of $\chi$-data, consisting of characters of $E_\lambda^\times$ satisfying the conditions similar to those of $\zeta$-data (see Section \ref{section Admissible embeddings of tori}).

With a collection of $\chi$-data, we follow \cite[Section 2.6]{LS} to construct an admissible embedding
$$I_{\{\chi_{\lambda,\xi}\}}:{}^LT\rightarrow {}^LG$$
 of the L-group ${}^LT$ of the maximal torus $T$ into the L-group ${}^LG$ of $G^*$. (For convenience, we call ${}^LT$ an L-torus in this paper.) Let $\tilde{\xi}:\mathcal{W}_F\rightarrow {}^LT$ be an L-homomorphism whose class is the parameter of the character $\xi$ of $E^\times=T(F)$ under the local Langlands correspondence of the torus $T$, i.e., the Artin reciprocity for $E^\times$ \cite{Tate-NTB}. In a previous result of the author \cite[Corollary 1.2]{thesis}, the Langlands parameter (\ref{paramater-as-induced-repres}) of ${}_F\Pi_\xi$ is isomorphic to the composition     $$I_{\{\chi_{\lambda,\xi}\}}\circ\tilde{\xi}:\mathcal{W}_F\rightarrow {}^LT\rightarrow {}^LG\xrightarrow{\text{natural proj.}}\mathrm{GL}_n(\mathbb{C})$$
    as a representation of $\mathcal{W}_F$. In other words, the essentially tame local Langlands correspondence for $G^*$ can be described by admissible embeddings of L-tori.

A set of $\zeta$-data is the `difference' of two sets of $\chi$-data, in the sense that, given a set $\{\chi_\lambda\}$  of $\chi$-data, we have
$$\{\chi'_\lambda\}\text{ is another set of }\chi\text{-data}\quad\Leftrightarrow\quad\{\chi_\lambda(\chi'_\lambda)^{-1}\}\text{ is a set of }\zeta\text{-data}.$$
If we define the local Langlands correspondence for $G$ as the composition of the local Langlands correspondence for $G^*$ and the Jacquet-Langlands correspondence $JL$, then we can express the Langlands parameter of ${}_D\Pi_\xi$ using an admissible embedding of L-tori, as follows.

\begin{cor}
  Let $\{\chi_{\lambda,\xi}\}$ and $\{\zeta_{\lambda,\xi}\}$ be respectively the
$\chi$-data and the $\zeta$-data associated to an admissible character
$\xi$. The Langlands parameter
  $${\mathrm{Ind}}_{\mathcal{W}_E}^{\mathcal{W}_F}(\xi\cdot{}_F\mu_\xi\cdot{}_D\nu_\xi)$$
  of ${}_D\Pi_\xi$ is isomorphic to
  $$I_{\{\chi_{\lambda,\xi}\cdot\zeta_{\lambda,\xi}\}}\circ\tilde{\xi}:\mathcal{W}_F\rightarrow
{}^LT\rightarrow {}^LG\xrightarrow{\text{natural proj.}}\mathrm{GL}_n(\mathbb{C})$$
  as a representation of $\mathcal{W}_F$.
\qed\end{cor}
Hence analogously we can describe the essentially tame local Langlands correspondence for $G$ by admissible embeddings of L-tori.

As a consequence, we show in Proposition \ref{rectifier decomp over K} that the factorization of ${}_D\nu_\xi$ in Theorem \ref{intro-main-thm-1}(\ref{zeta-data factorize}) is functorial, in the following sense. Let $K/F$ be an intermediate extension of $E/F$, so that if the pair $(E/F,\xi)$ is admissible over $F$, then $(E/K,\xi)$ is admissible over $K$ by definition. We denote the centralizer of $K^\times$ in $G(F)$ by $GL_{m_K}(D_K)$, where $D_K$ is a $K$-division algebra and $m_K$ a positive integer. If $\{\zeta_{\lambda,\xi}\}$ is the set of $\zeta$-data associated to $\xi$, then the partial product
$$
  \prod_{
    \lambda\in \Gamma_F\backslash \Phi,\,    \lambda|_{K^\times}\neq 1
  }\zeta_{\lambda,\xi}|_{E^\times}.
$$
(a product similar to Theorem \ref{intro-main-thm-1}.(\ref{zeta-data factorize}), with factors ranging over the characters being non-trivial on $K^\times$) is the rectifier ${}_{D_K}\nu_\xi$ of $\xi$ over $K$.

\begin{rmk}\label{characteristic}
As mentioned in \cite[Remark 1.3]{thesis}, we apply the condition char$(F)=0$ in this paper only because we refer to the theory of endoscopy. However, we do not actually need this condition for the part of the theory that we allude to, namely the admissible embedding and the transfer factor $\Delta_{\mathrm{III}_2}$, because these objects can be defined without any condition on char$(F)$. Therefore, we may treat this condition as a mild condition.
  \qed\end{rmk}

\subsection{Notations}\label{section Notations}

Throughout the paper, $F$ denotes a non-Archimedean local field of characteristic 0. Its ring of integers is $\mathfrak{o}_F$ with the maximal ideal $\mathfrak{p}_F$. The residue field $\mathbf{k}_F=\mathfrak{o}_F/\mathfrak{p}_F$ has $q$ elements and is of characteristic $p$. We denote by $v_F:F^\times\rightarrow \mathbb{Z}$ the discrete valuation on $F$. We denote by $\Gamma_F$ the absolute Galois group of $F$, and by $\mathcal{W}_F$ the Weil group of $F$.

The multiplicative group $F^\times$ decomposes into a product of subgroups
$$\left<\varpi_F\right>\times{\boldsymbol{\mu}}_F\times U_F^1 .$$ They are namely the group generated by a prime element $\varpi_F$, the group ${\boldsymbol{\mu}}_F$ of roots of unity of order prime to $p$, and the 1-unit group $U_F^1:=1+\mathfrak{p}_F$. We will identify ${\boldsymbol{\mu}}_F$ with $\mathbf{k}_F^\times$ in the natural way. We then write $U_F=U_F^0:={\boldsymbol{\mu}}_F\times U_F^1$ and $U_F^i:=1+\mathfrak{p}^i_F$ for each positive integer $i$. Let ${\boldsymbol{\mu}}_n$ be the group of $n$th roots of unity in the algebraic closure $\bar{F}$ of $F$, and ${{z}}_n$ be a choice of primitive $n$th root in ${\boldsymbol{\mu}}_n$.

The $F$-level of a character $\xi$ of $F^\times$ is the smallest integer $a\geq-1$ such that $\xi|_{U_F^{a+1}}$ is trivial. A character $\xi$ of $F^\times$ is called unramified if $\xi|_{U_F}$ is trivial, or equivalently, if its $F$-level is $-1$. It is called tamely ramified if $\xi|_{U_F^1}$ is trivial, or equivalently, if its
$F$-level is $0$.

Given a field extension $E/F$, we denote its ramification index by $e=e(E/F)$ and its residue degree by $f=f(E/F)$. We also denote by $tr_{E/F}$ and norm $N_{E/F}$ the trace and norm respectively.

 We fix an additive character $\psi_F$ of $F$ of level 0, which means that $\psi_F$ is trivial on $\mathfrak{p}_F$ but
is non-trivial on $\mathfrak{o}_F$. Hence $\psi_F |_{\mathfrak{o}_F}$ induces a non-trivial character of $\mathbf{k}_F$. We write $\psi_E=\psi_F\circ\mathrm{tr}_{E/F}$.

Suppose that $A$ is a central simple algebra over $F$. We denote the reduced trace by $\mathrm{trd}_{A/F}$ and the reduced norm by $\mathrm{Nrd}_{A/F}$.

 Given a set $X$, we denote its cardinality by $\#X$. If $H$ is a group and $X$ is a $H$-set, then we denote the action of $h\in H$ on $x\in X$ by $x\mapsto {}^hx$. The set of $H$-orbits is denoted by $H\backslash X$. If $\pi$ is a representation of $H$ (over a given field), we denote its equivalence class by $(H,\pi)$.

\section{Some basic setups}

\subsection{Root system}\label{section root system}

 Given a field extension $E/F$ of degree $n$, we let $T$ be the induced torus $\text{Res}_{E/F} \mathbb{G}_m$. We embed $T$ into $G$ as an elliptic maximal torus, and denote the image still by $T$. The choice of this embedding will be specific in Section \ref{subsection An embedding condition}, but at this moment this choice is irrelevant. Let $\Phi=\Phi(G,T)$ be the root system of $T$ in $G$. Following \cite[Section 3.1]{thesis}, we can denote each root in $\Phi$ by $\left[ \begin{smallmatrix}
g \\
h
\end{smallmatrix} \right]$ where $g=g\Gamma_E$ and $h=h\Gamma_E$ are distinct cosets in $\Gamma_F/\Gamma_E$. (We use the same notation $g$ for an element in $\Gamma_F$ and its $\Gamma_E$-cosets, for notation convenience.) The $\Gamma_F$-action on $\Phi$ is given by
$ x \cdot \left[ \begin{smallmatrix}
 g\\
 h
\end{smallmatrix}
\right] = \left[ \begin{smallmatrix}
xg\\
xh
\end{smallmatrix}\right].$
 For each root $\lambda\in\Phi$, we denote by $[\lambda]$ its $\Gamma_F$-orbit. Each $\Gamma_F$-orbit contains a root of the form
$\left[\begin{smallmatrix}1\\g\end{smallmatrix}\right]$ for some non-trivial coset $g\in \Gamma_F/\Gamma_E$.

For each root $\lambda \in \Phi$, we denote the stabilizers $\{ g \in \Gamma_F | {}^g \lambda = \lambda\}$ and $\{ g \in \Gamma_F | {}^g\lambda = \pm \lambda \}$ by $\Gamma_{\lambda} $ and $\Gamma_{\pm \lambda}$ respectively and their fixed fields by $E_{\lambda}$ and $ E_{\pm \lambda}$ respectively. We call a root  $\lambda$ \emph{symmetric} if $[E_{\lambda}:E_{\pm\lambda}] = 2$, and \emph{asymmetric} otherwise. Equivalently, $\lambda$ is symmetric if and only if $\lambda$ and $-\lambda$ are in the same $\Gamma_F$-orbit. Note that the symmetry of $\Phi$ is preserved by the $\Gamma_F$-action. Let
\begin{enumerate}[(i)]
\item $ \Gamma_F\backslash\Phi_{\mathrm{sym}}$ be the set of $\Gamma_F$-orbits of symmetric roots,
\item $ \Gamma_F\backslash\Phi_\mathrm{asym}$ be the set of $\Gamma_F$-orbits of asymmetric roots, and
\item $ \Gamma_F\backslash\Phi_\mathrm{asym/\pm}$ be the set of equivalence classes of asymmetric $\Gamma_F$-orbits by identifying $[\lambda]$ and $[-\lambda]$.
\end{enumerate}

We denote by $(\Gamma_E\backslash \Gamma_F/\Gamma_E)'$ the collection of non-trivial double cosets, and by $[g]$ the double coset $\Gamma_Eg\Gamma_E$. We can deduce the following proposition easily.

\begin{prop}\label{orbit of roots as double coset}
The map
$$\Gamma_F\backslash\Phi\rightarrow (\Gamma_E\backslash \Gamma_F/\Gamma_E)',\,[\lambda]=\left[\left[\begin{smallmatrix}1\\g\end{smallmatrix}\right]\right]\mapsto [g],$$
is a bijection between the set $\Gamma_F\backslash\Phi$ of $\Gamma_F$-orbits of the root system $\Phi$ and the set $(\Gamma_E\backslash \Gamma_F/\Gamma_E)'$ of non-trivial double cosets.
\end{prop}

We can therefore call $g\in \Gamma_F$ \emph{symmetric} if $[g]= [g^{-1}]$, and \emph{asymmetric} otherwise, so that the bijection in  Proposition \ref{orbit of roots as double coset} preserves symmetries on both sides. Let \begin{enumerate}[(i)]
\item $(\Gamma_E\backslash \Gamma_F/\Gamma_E)_\mathrm{sym}$ be the set of symmetric non-trivial double cosets,
\item $(\Gamma_E\backslash \Gamma_F/\Gamma_E)_\mathrm{asym}$ be the set of asymmetric non-trivial double cosets, and
\item $(\Gamma_E\backslash \Gamma_F/\Gamma_E)_{\mathrm{asym}/\pm}$ be the set of equivalence classes of $(\Gamma_E\backslash \Gamma_F/\Gamma_E)_\mathrm{asym}$ by identifying $[g]$ with $[g^{-1}]$.
\end{enumerate}

We choose subsets $\mathcal{D}_\mathrm{sym}$  and $ \mathcal{D}_{\mathrm{asym}/ \pm}$ of representatives in $\Gamma_F/\Gamma_E$ of $(\Gamma_E\backslash \Gamma_F/\Gamma_E)_\mathrm{sym}$ and  $(\Gamma_E\backslash \Gamma_F/\Gamma_E)_{\mathrm{asym}/\pm}$ respectively, and write
$$\mathcal{D}_\mathrm{asym}=\mathcal{D}_{\mathrm{asym}/ \pm}\sqcup\{g^{-1}|g\in \mathcal{D}_{\mathrm{asym}/ \pm}\}.$$
We also choose subsets $\mathcal{R}_\mathrm{sym}$ and $ \mathcal{R}_{\mathrm{asym}/ \pm}$  of representatives in $\Phi$ of orbits in
$ \Gamma_F\backslash\Phi_{\mathrm{sym}}$  and $ \Gamma_F\backslash\Phi_\mathrm{asym/\pm}$ respectively such that every root $\lambda\in\mathcal{R}_\pm:=\mathcal{R}_\mathrm{sym}\sqcup\mathcal{R}_{\mathrm{asym}/ \pm}$ is of the form $\left[\begin{smallmatrix}1\\g\end{smallmatrix}\right]$ for some $g\in\mathcal{D}_\pm:=\mathcal{D}_\mathrm{sym}\sqcup\mathcal{D}_{\mathrm{asym}/ \pm}$, and write $$\mathcal{R}_\mathrm{asym}= \mathcal{R}_{\mathrm{asym}/ \pm}\sqcup( -\mathcal{R}_{\mathrm{asym}/ \pm}).$$
Hence  $\mathcal{R}_\mathrm{sym}$, $\mathcal{R}_\mathrm{asym}$, and $ \mathcal{R}_{\mathrm{asym}/ \pm}$ correspond bijectively to $\mathcal{D}_\mathrm{sym}, \mathcal{D}_\mathrm{asym}$, and $ \mathcal{D}_{\mathrm{asym}/ \pm}$ respectively by the identification in Proposition \ref{orbit of roots as double coset}. Denote $E_g:=E_\lambda $ and $E_{\pm g}:=E_{\pm \lambda} $. Notice that $E_g= E({}^g E )$, composite field of $E$ and ${}^gE$.

\subsection{Galois groups}\label{section Galois groups}

 Let $E/F$ be a field extension of degree $n$. In most of the paper, we assume that $E/F$ is tamely ramified, which means that $p$ is coprime to $e$. By \cite[II.\S5]{Lang-ANT}, we can choose $\varpi_E$ and $\varpi_F$ such that
 \begin{equation}\label{e-th power of prime is also prime}
 \varpi_E^e={{z}}_{E/F}\varpi_F\text{, for some }{{z}}_{E/F}\in {\boldsymbol{\mu}}_E.
 \end{equation}

Choose in $\bar{F}^\times$ a primitive $e$th root of unity ${{z}}_e$ and an $e$th root ${{z}}_{E/F,e}$ of ${{z}}_{E/F}$. (We do not require that ${{z}}_{E/F,e}^a={{z}}_e$, if $a$ is the multiplicative order of ${{z}}_{E/F}$.) Denote $L=E[{{z}}_e,{{z}}_{E/F,e}]$ and $l=[L:E]$. With the choices of $\varpi_F$ and $\varpi_E$ as in (\ref{e-th power of prime is also prime}), we define the following $F$-automorphisms on $L$.
\begin{enumerate}[(i)]
\item $\phi: {{z}} \mapsto {{z}}^q \text{, for all }{{z}} \in {\boldsymbol{\mu}}_L\text{, and }\phi:\varpi_E \mapsto {{z}}_\phi\varpi_E$.
\item $\sigma : {{z}} \mapsto {{z}}\text{, for all }{{z}} \in {\boldsymbol{\mu}}_L\text{, and }\sigma:\varpi_E \mapsto {{z}}_e \varpi_E.$
\end{enumerate}
Here ${{z}}_\phi$ lies in ${\boldsymbol{\mu}}_E$ satisfying $({{z}}_\phi\varpi_E)^e={{z}}_{E/F}^q\varpi_F$. More generally, we write $^{\phi^i}\varpi_E = {{z}}_{\phi^i}\varpi_E$ where ${{z}}_{\phi^i}={{z}}_\phi^{1+q+\cdots+q^{i-1}}$ is an $e$th root of ${{z}}_{E/F}^{q^i-1}$.

Therefore, $\Gamma_{L/F} = \langle \sigma \rangle \rtimes \langle \phi \rangle $ with relation $\phi \circ \sigma \circ \phi^{-1}= \sigma^q$. Suppose that $\Gamma_{L/E} = \langle \sigma^c\phi^f \rangle $ for some integer $c$ satisfying the condition:
\begin{equation*}
  e\text{ divides }c\left(\frac{q^{fl}-1}{q^f-1}\right).
\end{equation*}
We can choose
$$\{\sigma^i\phi^j|i=0,\dots,e-1,\,j=0,\dots,f-1\}$$
as coset representatives for the quotient $\Gamma_{E/F}= \Gamma_F/\Gamma_E$. Moreover, elements in a fixed double coset are of the form $[\sigma^i\phi^j]$ with a fixed $j$ mod $f$.

\begin{prop}[{\cite[Proposition 3.3]{thesis}}]\label{properties of symmetric [g]}
  The double coset $[g]=[\sigma^{i}\phi^j] $ is symmetric only if $j=0$ or, when $f$ is even, $j=f/2$.
\qed\end{prop}
We call those symmetric $[\sigma^{i}]$ ramified and those symmetric $[\sigma^{i}\phi^{f/2}]$ unramified, and denote by $(\Gamma_E\backslash \Gamma_F/  \Gamma_E)_{\mathrm{sym-ram}}$ and $(\Gamma_E\backslash \Gamma_F/  \Gamma_E)_{\mathrm{sym-unram}}$ respectively the collections of symmetric ramified and symmetric unramified double cosets.

We provide several useful results concerning the parity of certain subsets in $\Gamma_E\backslash \Gamma_F /\Gamma_E$.

 \begin{prop}[{\cite[Propositions 3.4 and 3.5]{thesis}}]\label{parity of double coset with i=f/2}
 \begin{enumerate}[(i)]
   \item If $[g]$ is symmetric unramified, then the degree $[E_g:E]$ is odd.
   \item The parity of  $\#(\Gamma_E\backslash \Gamma_F/  \Gamma_E)_{\mathrm{sym-unram}}$ is equal to that of ${e(f-1)}$.
 \end{enumerate}
\qed\end{prop}

\begin{lem}\label{exists root fixing varpi}
  Suppose that $f$ is even. The following are equivalent.
  \begin{enumerate}[(i)]
    \item There exists $\sigma^i \phi^{f/2} \in W_{F[\varpi_E]}$ for some $i$.
      \label{exists root fixing varpi exist}
    \item  ${{z}}_{\phi^{f/2}} $ is an $e$th root of unity. \label{exists root fixing varpi root}
    \item  ${{z}}_{E/F} \in K_+$, where $K_+/F$ is unramified of degree $f/2$. \label{exists root fixing varpi quad}
    \item $f_\varpi :=[E:F[\varpi_E]]$ is even. \label{exists root fixing varpi even}
  \end{enumerate}
\end{lem}

\begin{proof}(\ref{exists root fixing varpi exist}) is equivalent to (\ref{exists root fixing varpi root}) since ${}^{\sigma^i \phi^{f/2}}\varpi_E={{z}}^i_e {{z}}_{\phi^{f/2}}\varpi_E$. To show that (\ref{exists root fixing varpi quad}) implies (\ref{exists root fixing varpi root}), we recall that ${{z}}_{\phi^{f/2}}$ is an $e$th root of ${{z}}^{q^{f/2}-1}_{E/F}$. If ${{z}}_{E/F} \in K_+$, then ${{z}}^{q^{f/2}-1}_{E/F}=1$ and ${{z}}_{\phi^{f/2}}$ is an $e$th root of unity. The converse is similar. To show the equivalence of (\ref{exists root fixing varpi quad}) and (\ref{exists root fixing varpi even}), we notice that $f(F[\varpi_E]/F) = f(F[{{z}}_{E/F}]/F) = f/f_\varpi$. Hence that $F[{{z}}_{E/F}] \subseteq K_+ $ is equivalent to that $ f_\varpi$ is even.
\end{proof}

\begin{lem}\label{lemma-second}
  Suppose that $g=\sigma^i \phi^{f/2}$ satisfies the conditions in Lemma \ref{exists root fixing varpi}.
  \begin{enumerate}[(i)]
    \item The double coset $[\sigma^i \phi^{f/2}]$ is automatically symmetric. \label{lemma-second-auto-sym}
    \item The set $(\Gamma_E\backslash \Gamma_{F[\varpi_E]}/\Gamma_E)_{\mathrm{sym-unram}}=(\Gamma_E\backslash \Gamma_{F}/\Gamma_E)_{\mathrm{sym-unram}}\cap (\Gamma_E\backslash \Gamma_{F[\varpi_E]}/\Gamma_E )$  contains a single element $[\sigma^i \phi^{f/2}]$. \label{lemma-second-single}
  \end{enumerate}
\end{lem}
\begin{proof}
  For (\ref{lemma-second-auto-sym}), we consider the actions of $\sigma^i \phi^{f/2}$ and its inverse $(\sigma^i \phi^{f/2})^{-1}$ on $E$. We certainly have ${}^{\sigma^i \phi^{f/2}}\varpi_E={}^{({\sigma^i \phi^{f/2}})^{-1}}\varpi_E=\varpi_E$ by definition. We also have ${}^{\sigma^i \phi^{f/2}}{{z}}={}^{({\sigma^i \phi^{f/2}})^{-1}}{{z}}={{z}}^{q^{f/2}}$ for all $z\in \boldsymbol{\mu}_E$. Therefore $\sigma^i \phi^{f/2}\Gamma_E=(\sigma^i \phi^{f/2})^{-1}\Gamma_E$ and in particular $[\sigma^i \phi^{f/2}]=[(\sigma^i \phi^{f/2})^{-1}]$. For (\ref{lemma-second-single}), we know by Lemma \ref{exists root fixing varpi}.(\ref{exists root fixing varpi root}) that the double coset is the one containing $\sigma^i \phi^{f/2}$ where $ {{z}}_{\phi^{f/2}}={{z}}^{-i}_e $.
\end{proof}

\begin{prop}\label{parity sym unram not fixing varpiE}
When $f$ is even, the parity of the cardinality of $(\Gamma_E\backslash \Gamma_F /\Gamma_E)_{\mathrm{sym-unram}}- \Gamma_{F[\varpi_E]}$ is equal to ${e+f_\varpi-1}$.
\end{prop}
\begin{proof}
Recall
\begin{enumerate}[(i)]
  \item by Lemma \ref{exists root fixing varpi} that there exists $\sigma^i \phi^{f/2} \in W_{F[\varpi_E]}$ if and only if $f_\varpi$ is even, and
      \item by Proposition \ref{parity of double coset with i=f/2} that the parity of the number of symmetric $[\sigma^i\phi^{f/2}]$ is the same as that of $e$.
\end{enumerate}
By combining these facts, we have the assertion.
\end{proof}

\subsection{Division algebra}\label{Conjugate relations in a division algebra}

Let $D$ be a division algebra over $F$ of dimension $n^2$. Denote its unique maximal order by $\mathfrak{o}_D$ and the maximal ideal of $\mathfrak{o}_D$ by $\mathfrak{p}_D$. Suppose that the Hasse-invariant of $D$ is $h=h(D)$, so that  $\gcd(n,h)=1$. By \cite[(14.5) Theorem]{reiner-max-order}, we can choose a primitive $(q^n-1)$th root ${z}$ of unity in $D$ and a uniformizer $\varpi_D$ such that
\begin{equation}\label{ad-varpi-D-on-zeta}
  \varpi_D^n=\varpi_F\text{ and }\varpi_D{z}\varpi_D^{-1}={z}^{q^h}.
\end{equation}
We write $K_n=F[z]$ and ${\boldsymbol{\mu}}_D=\left<z\right>$, then $K_n$ is a maximal unramified extension (of degree $n$) in $D$, and ${\boldsymbol{\mu}}_D$ is a group of roots of unity of order $q^n-1$, both defined up to conjugacy by $D^\times$. Therefore, the conjugation of $\varpi_D$ acts on $K_n$ as the $h$th power of the Frobenius automorphism, i.e.,
$$\varpi_Du\varpi_D^{-1}={}^{\phi^h}u\text{ for all }u\in K_n.$$
We write $U_D^i:=1+\mathfrak{p}^i_D$ for all positive integer $i$. The multiplicative subgroup $D^\times$ hence decomposes into a semi-direct product
$$(\left<\varpi_D\right>\ltimes {\boldsymbol{\mu}}_D)\ltimes U_D^1.$$

Let $E/F$ be a tamely ramified field extension of degree $n$, and let $K$ be the maximal unramified sub-extension in $E/F$. We assume that $K\subseteq K_n$ and that the uniformizers $\varpi_E$ and $\varpi_F$ satisfy $\varpi_E^e={z}_{E/F}\varpi_F$ as in (\ref{e-th power of prime is also prime}) for some ${z}_{E/F}\in {\boldsymbol{\mu}}_E$. If we define ${z}_{D/E} \in {\boldsymbol{\mu}}_D= {\boldsymbol{\mu}}_{K_n}$ to be a solution of
\begin{equation}\label{norm of root of unity}
  N_{K_n/K}({z}_{D/E})={z}_{E/F},
\end{equation}
then we may take $ \varpi_E=\varpi_D^f{z}_{D/E} $ and this defines an embedding $E$ into $D$ over $F$. Note that from (\ref{ad-varpi-D-on-zeta})
\begin{equation}\label{ad-zeta-on-varpi-D-i}
 {z} \varpi_D^i{z}^{-1}={z}^{1-q^{hi}}\varpi_D^i
\end{equation}
for all $z\in {\boldsymbol{\mu}}_E={\boldsymbol{\mu}}_K$ and all $i\in\mathbb{Z}$.

\subsection{Hereditary orders in central simple algebra}\label{section hereditary order in CSA}

If $G$ is an $F$-inner form of $G^*={\mathrm{GL}}_n$, then $G(F)=A^\times$, where $A$ be a central simple algebra over $F$. By Wedderburn Theorem \cite[(7.4)Theorem]{reiner-max-order}, $A$ is isomorphic to $\mathrm{Mat}_m(D)$, where $D$ is a division algebra of $F$-dimension $d^2$ and $md=n$. Therefore, $G(F)\cong {\mathrm{GL}}_m(D)$. Any field extension of degree $n$ can be embedded into $A$ as a maximal subfield in $A$, and any two such embeddings are conjugate under $G(F)$.

Let $\mathfrak{A}$ be an $\mathfrak{o}_F$-hereditary order in $A$, $\mathfrak{P_A}$ be its Jacobson radical, and $\mathfrak{K_A}$ be the $G(F)$-normalizer of $\mathfrak{A}^\times$. If $\mathfrak{A}$ is principal, in the sense that there exists $\varpi_\mathfrak{A}\in \mathfrak{K_A}$ such that $\varpi_\mathfrak{A}\mathfrak{A}=\mathfrak{A}\varpi_\mathfrak{A}=\mathfrak{P_A}$, then the valuation $v_\mathfrak{A}:\mathfrak{K_A}\rightarrow \mathbb{Z}$ is defined by $x\mathfrak{A}=\mathfrak{A}x=\mathfrak{P_A}^{v_\mathfrak{A}(x)}$ for all $x\in \mathfrak{K_A}$. We also write $U_\mathfrak{A}=U^0_\mathfrak{A}=\mathfrak{A}^\times $, $U^i_\mathfrak{A}=1+\mathfrak{P_A}^i $ for each positive integer $i$, $U^x_\mathfrak{A}=U^{\lceil x\rceil}_\mathfrak{A}$ for all $x\in \mathbb{R}_{\geq0}$, and $$U^{x+}_\mathfrak{A}=\bigcup_{x\in \mathbb{R}_{\geq0},\, y>x}U^{y}_\mathfrak{A}.$$

Suppose that ${E_0}$ is a subfield in $A$ and $\mathfrak{A}$ is ${E_0}$-pure, i.e., ${E_0}^\times\subseteq \mathfrak{K_A}$, then we define the ramification index $e(\mathfrak{A}/\mathfrak{o}_{{E_0}})$ to be the integer $e$ satisfying $v_\mathfrak{A}|_{{E_0}^\times}=ev_{E_0}$. We therefore have
\begin{equation*}\label{small-in-big}
  \mathfrak{p}_{E_0}^i\mathfrak{P_A}^j=\mathfrak{P_A}^{ie+j}
\end{equation*}
and
\begin{equation*}\label{small-intersect-big}
\mathfrak{p}_{E_0}^i\cap \mathfrak{P_A}^j=\mathfrak{p}_{E_0}^{\mathrm{max}\{i, j/e\}}
\end{equation*}
for all $i,j\in \mathbb{Z}$.

In the split case $G=G^*$, i.e., when $A=M$, we denote the hereditary order $\mathfrak{A}$ by $\mathfrak{M}$.

\subsection{Embedding conditions}\label{subsection An embedding condition}

Let $E_0$ and $\mathfrak{A}$ as in the previous section, and write $A_0$ the centralizer of $E_0$ in $A$. In this paper, we assume the conditions (\ref{3-condition-E-pure})-(\ref{3-condition-cano-cont}) below.
\newcounter{saveenum}
\begin{enumerate}[(i)]
 \item \cite[Section 3.2]{BH-JLC} If $E/E_0$ is an unramified extension in $A_0$ such that $[E:F]=n$, then we require that $\mathfrak{A}$ is $E$-pure, i.e., $E^\times\subseteq \mathfrak{K_A}$. \label{3-condition-E-pure}
  \setcounter{saveenum}{\value{enumi}}
\end{enumerate}
Let ${\mathfrak{A}_0}$ be the centralizer of $E_0$ in ${\mathfrak{A}}$, i.e., ${\mathfrak{A}_0}=\mathfrak{A}\cap A_0$, which is a hereditary $\mathfrak{o}_{E_0}$-order in $A_0$, with Jacobson radical $\mathfrak{P}_{\mathfrak{A}_0}=\mathfrak{P}_\mathfrak{A}\cap A_0$.
\begin{enumerate}[(i)]
  \setcounter{enumi}{\value{saveenum}}
 \item \cite[(2.3.1)(2)]{BH-JLC} There exists a fixed integer $e(\mathfrak{A}/\mathfrak{A}_0)\geq 1$ such that
 \begin{equation}\label{cano-cont-intersection}
   \mathfrak{P}^k_\mathfrak{A}\cap A_0=\mathfrak{P}_{\mathfrak{A}_0}^{ k/e(\mathfrak{A}/\mathfrak{A}_0)}\text{ for every }k\in \mathbb{Z}.
 \end{equation}
\label{3-condition-intersection}
  \setcounter{saveenum}{\value{enumi}}
\end{enumerate}
We say that $\mathfrak{A}$ is the \emph{canonical continuation} of $\mathfrak{A}_0$ in $A$. Under (\ref{3-condition-intersection}), we have moreover $\mathfrak{K_A}\cap A_0=\mathfrak{K}_{\mathfrak{A}_0}$.
\begin{enumerate}[(i)]
  \setcounter{enumi}{\value{saveenum}}
   \item \cite[(2.3.2)]{BH-JLC}
   $\mathfrak{A}_0$ is a maximal hereditary $\mathfrak{o}_{E_0}$-order in $A_0$, i.e., $e(\mathfrak{A}_0/\mathfrak{o}_{E_0})=1$.
  \label{3-condition-cano-cont}
\end{enumerate}
  Under (\ref{3-condition-cano-cont}), then $\mathfrak{A}$ is the unique $E_0$-pure hereditary order in $A$ such that $\mathfrak{A}\cap A_0={\mathfrak{A}_0}$. Moreover, $\mathfrak{A}$ is maximal among all $E_0$-pure hereditary orders in $A$, and both $\mathfrak{A}$ and $\mathfrak{A}_0$ are principal (by \cite[the remark after (2.3.2)]{BH-JLC}).

 By \cite[0. Theorem]{Zink}, if the $\mathfrak{o}_D$-period of $\mathfrak{A}$ is denoted by $r=r(\mathfrak{A})=e(\mathfrak{A}/\mathfrak{o}_D)$, i.e. $\varpi_D\mathfrak{A}=\mathfrak{P_A}^r$, then we have an isomorphism
\begin{equation}\label{A mod PA}
  \mathfrak{A}/\mathfrak{P_A}\cong \mathrm{Mat}_s(\mathbf{k}_D)^r,
\end{equation}
where $s=s(\mathfrak{A})=f/e(\mathfrak{A}/\mathfrak{o}_E)$. Once $A$ (and hence $D$) is fixed, the integers $r$ and $s$ depend only on $E$; indeed
$s=s(E/F)=\gcd(f,m)$ and $r=r(E/F)=e/\gcd(d,e)=m/s$.

If $K$ is an intermediate subfield in $E/F$, we write $f_K=f(E/K)$ and $e_K=e(E/K)$. By \cite[1. Prop.]{Zink} the centralizer $A_K=Z_A(K)$ is isomorphic to $\mathrm{Mat}_{m_k}(D_K)$, where $D_K$ is a division algebra over $K$ of degree $d_K^2$, with
$$d_K=\frac{d}{\gcd(d,n(K/F))}\text{ and }m_K=\gcd(m,{n(E/K)}).$$
Let $\mathfrak{A}_K$ be the centralizer of $K$ in $\mathfrak{A}$. It is routine to check that $(E_0K,\mathfrak{A}_K)$ is a canonical continuation of $\mathfrak{A}_{E_0K}$ and is also $E$-pure, i.e., (\ref{3-condition-E-pure})-(\ref{3-condition-cano-cont}) are satisfied when we change our base field from $F$ to $K$. We therefore have an isomorphism
\begin{equation}\label{A mod PA}
  \mathfrak{A}_K/\mathfrak{P}_{\mathfrak{A}_K}\cong \mathrm{Mat}_{s_K}(\mathbf{k}_{D_K})^{r_K},
\end{equation}
where $r_K=e_K/\gcd(d_K,e_K)$ and $s_K=\gcd(f_K,m_K)$.

\subsection{Characters}

Let $\mathrm{Nrd}_{A/F}:A^\times\rightarrow F^\times$ be the reduced norm of $G(F)=A^\times$. By \cite[(33.4)Theorem]{reiner-max-order} and \cite[Satz 2]{Nakayama-Matsushima}, $\mathrm{Nrd}_{A/F}$ is surjective and its kernel is the commutator subgroup of $G(F)$. Therefore, any character of $G(F)$ factors through $\mathrm{Nrd}_{A/F}$.

We define the $\mathfrak{A}$-level of a character $\xi$ of $A^\times$ as the smallest integer $a\geq-1$ such that $\xi|_{U_\mathfrak{A}^{a+1}}$ is trivial. This is analogous to the $F$-level of a character of $F^\times$ defined in Section \ref{section Notations}.

\begin{prop}\label{char composing norm}
 If the $F$-level of a character $\xi$ of $F^\times$ is $r$, then the $\mathfrak{A}$-level of the character $\xi\circ \mathrm{Nrd}_{A/F}$ is $r\cdot e(\mathfrak{A}/\mathfrak{o}_F)$.
\end{prop}

\begin{proof}
  It suffices to show that
  $$\mathrm{Nrd}_{A/F}({U_\mathfrak{A}^{r\cdot e(\mathfrak{A}/\mathfrak{o}_F)}})=U_F^r
  \text{ for all }r\in \mathbb{R}_{\geq 0}.$$
  This follows from \cite[(2.8.3)]{BF-non-ab-Gauss-sum}.
\end{proof}

\section{The essentially tame Jacquet-Langlands correspondence}

\subsection{Admissible characters}\label{section Admissible characters}

We recall the definition of admissible characters in \cite{Howe}, \cite{Moy}. Given a tamely ramified finite extension $E/F$, let $\xi$ be a character of $E^\times$. We call a pair $(E/F,\xi)$ \emph{admissible} \index{admissible character} over $F$ if $\xi$ is an admissible character over $F$, i.e., for every intermediate subfield $K$ between $E/F$,
\begin{enumerate}[(i)]
\item if $\xi$ factors through the norm $ N_{E/K}$, then $E = K$;
\item if $\xi|_{U^1_E}$ factors through $ N_{E/K}$ then $E/K$ is unramified.
\end{enumerate}
Two admissible pairs $(E/F,\xi)$ and $(E'/F,\xi')$ are called $F$-equivalent if there is $g\in \Gamma_F$ such that ${}^gE=E'$ and ${}^g\xi=\xi'$. Let $P_n(F)$ be the set of $F$-equivalence classes of the admissible pair $(E/F,\xi)$, with each class in $P_n(F)$ still denoted by $(E/F,\xi)$ for convenience.

Every admissible character $\xi$ admits a {factorization} (see \cite[Corollary of Lemma 11]{Howe} or \cite[Lemma 2.2.4]{Moy})
\begin{equation}\label{Howe factorization}
\xi=  \xi_{-1}(\xi_0 \circ N_{E/E_0})\cdots(\xi_t \circ N_{E/E_t})(\xi_{t+1}\circ N_{E/F}),
\end{equation}
where, in the notations above, the decreasing sequence of fields
\begin{equation*}\label{decreasing subfield}
E=E_{-1}\supseteq E_0 \supsetneq E_1 \supsetneq \cdots \supsetneq E_t \supsetneq E_{t+1}=F.
 \end{equation*}
 and the increasing $E$-levels
 $a_{-1}=0<a_0<a_1<\cdots<a_t\leq a_{t+1}$
 of the characters $\xi_k \circ N_{E/E_k}$, $k=0,\dots,t+1$, are uniquely determined. We call the $E$-levels $a_k$ the \emph{jumps} of $\xi$ and call the collection $\{E_k,a_k|k=0,\dots,t\}$ the \emph{jump data} of $\xi$. By convention, when $E_0=E$, we replace $(\xi_0 \circ N_{E/E_0})\xi_{-1}$ by $\xi_0$ and assume that $\xi_{-1}$ is trivial; otherwise, $\xi_{-1}$ is tamely ramified and $E/E_0$ is unramified \cite[Defnition 2.2.3]{Moy}. There are some conditions imposed on the jump data of the character by its admissibility, but we do not need them in this paper. We only refer the interested reader to \cite[Section 8.2]{BH-ET3} (and when $E/F$ is totally ramified, see also \cite[Section 1]{BH-ET2}).

We fix a (non-canonical) choice of $\xi_{-1}$ in the factorization (\ref{Howe factorization}) as follows. We fix a choice of the wild component $\xi_w$ of $\xi$ to be the product
$$(\xi_0 \circ N_{E/E_0})\cdots(\xi_t \circ N_{E/E_t})(\xi_{t+1}\circ N_{E/F})$$
which satisfies
\begin{equation}\label{xi-1 conditions}
  \xi_w(\varpi_F)=1\text{ and }\xi_w\text{ has a }p\text{-power order}
\end{equation}
(see \cite[Lemma 1 of Section 4.3]{BH-JLC}), and define the tame component of $\xi$ to be $\xi_{-1} = \xi \xi_w^{-1}$.
We write
\begin{equation}\label{factorization of 1-pair}
\Xi=\Xi(\xi) =  \xi_0  (\xi_1 \circ N_{E_0/E_1})  \cdots(\xi_{t+1} \circ N_{E_0/F}),
\end{equation}
such that $\xi_w=\Xi \circ N_{E/E_0}$.

Suppose that  $E_0/F$ is a tamely ramified extension of degree dividing $n$ and $E/E_0$ is unramified. Let $\Xi$ be a character of $U_{E_0}^1$. We call $(E_0/F,\Xi)$ an \emph{admissible 1-pair} over $F$ if $\Xi$ does not factor through any norm $N_{E_0/K}$ with $F\subseteq K\subsetneq E_0$. We denote by $P^1_n(F)$ the set of $F$-equivalence classes of these pairs. Therefore, the map
$$P_n(F)\rightarrow P^1_n(F),\,(E/F,\xi)\mapsto (E_0/F,\Xi(\xi)|_{U_{E_0}^1}),$$
is well-defined and surjective. Notice that we can define the jump-data of a 1-pair, such that the jump-data of an admissible pair is the same as that of its associated 1-pair.

\subsection{The correspondences}\label{section The correspondences}

Let $G^*$ be ${\mathrm{GL}}_n$ defined over $F$, and $G$ be an inner form of $G^*$ whose $F$-point is isomorphic to ${\mathrm{GL}}_m(D)$ for some central $F$-division algebra $D$ of dimension $d^2$ and $n=md$. Let
$\mathcal{A}^2_n(F)$ (resp. $\mathcal{A}^2_m(D)$) be the collection of equivalence classes of irreducible smooth representations of $G^*(F)$ (resp. $G(F)$) which are essentially square-integrable modulo center. The Jacquet-Langlands correspondence is a canonical bijection
\begin{equation}\label{JLC}
  JL:\mathcal{A}^2_n(F)\rightarrow \mathcal{A}^2_m(D)
\end{equation}
between the two collections determined by a character relation between  $\pi\in \mathcal{A}^2_n(F)$ and its image $JL(\pi)\in \mathcal{A}^2_m(D)$: for every pair of semi-simple elliptic regular elements $(g,g^*)$, where $g\in G(F)$ and $g^*\in G^*(F)$, with the same reduced characteristic polynomial, we have \cite[Section 1.4]{BH-JLC}
\begin{equation}\label{JL-character-relation}
  (-1)^{n-m}\Theta_{\pi}(g^*)=\Theta_{JL(\pi)}(g),
\end{equation}
where $\Theta_{\pi}$ (resp. $\Theta_{JL(\pi)}$) is the character of $\pi$ (resp. ${JL(\pi)}$).

For each representation $\pi\in \mathcal{A}^2_m(D)$, let
\begin{enumerate}[(i)]
  \item   $f(\pi)$ be the number of unramified characters $\chi$ of $F^\times$ that $\chi\otimes\pi\cong\pi$ (here $\chi$ is regarded as a representation of $G(F)$ by composing with the reduced norm map $\mathrm{Nrd}:G(F)\rightarrow F^\times$), and
      \item $\delta(\pi)$ be the parametric degree of $\pi$ (we do not require its full definition, so we only refer to \cite[Section 2.7]{BH-JLC} for details).
\end{enumerate}
It is known that $\delta(\pi)$ is a positive integer and is a multiple of $f(\pi)$. Moreover, $\pi$ is supercuspidal if $\delta(\pi)=n$, while the converse is only true in the split case (when $G=G^*$).

The correspondence (\ref{JLC}) restricts to a bijection
\begin{equation*}
  JL:\mathcal{A}^0_n(F)\rightarrow \{\pi\in \mathcal{A}^2_m(D)| \delta(\pi)=n\}.
\end{equation*}
We call $\pi$ \emph{essentially tame} if $p$ does not divide $\delta(\pi)/f(\pi)$. Let $\mathcal{A}^\mathrm{et}_m(D)$ be the set of isomorphism classes of irreducible representations in $\mathcal{A}^2_m(D)$ which are essentially tame and satisfy $\delta(\pi)=n$.  Therefore $\mathcal{A}^\mathrm{et}_n(F)$ is the same collection defined in \cite{BH-ET1}. Since the Jacquet-Langlands correspondence in (\ref{JLC}) preserves the invariants $\delta(\pi)$ and $f(\pi)$, we have the following theorem \cite[2.8. Corollary 2]{BH-JLC}.

 \begin{thm}[Essentially tame Jacquet-Langlands correspondence]\label{ETJLC}
The restriction of the Jacquet-Langlands correspondence induces a bijection
$$JL:\mathcal{A}^\mathrm{et}_n(F)\rightarrow\mathcal{A}^\mathrm{et}_m(D).$$
 This bijection preserves the central characters on both sides.
 \qed\end{thm}

Bushnell and Henniart described this bijection explicitly in a way parallel to \cite{BH-ET1}, \cite{BH-ET2}, \cite{BH-ET3}. We recall the results briefly as follows. On the one hand, we have the bijection
\begin{equation}\label{bijection of P and FPi}
  {}_F\Pi:P_n(F)\rightarrow\mathcal{A}^\mathrm{et}_n(F),\,(E/F,\xi)\mapsto{}_F\Pi_\xi,
\end{equation}
generalizing the construction of Howe \cite{Howe}. On the other hand, there is an analogous bijection
\begin{equation}\label{bijection of P and DPi}
  {}_D\Pi:P_n(F)\rightarrow\mathcal{A}^\mathrm{et}_m(D),\,(E/F,\xi)\mapsto{}_D\Pi_\xi,
\end{equation}
using the constructions in \cite{Secherre1}, \cite{Secherre2}, \cite{Secherre3}, \cite{Secherre4}. In fact, what is constructed in \cite{BH-JLC} is the inverse of (\ref{bijection of P and DPi}), using the method called `attached-pairs', since the construction parallel to (\ref{bijection of P and FPi}) exhibits some `novel technical difficulties' as mentioned in \cite[Introduction 4.]{BH-JLC}. In the split case, the attached-pair method yields the inverse of (\ref{bijection of P and FPi}) (see 4.4 of \cite{BH-JLC}).

The composition of the bijection in (\ref{bijection of P and FPi}), the correspondence in Theorem \ref{ETJLC}, and the inverse of (\ref{bijection of P and DPi}),
\begin{equation}\label{composition is not identity}
{}_D\nu: P_n(F) \xrightarrow{{}_F\Pi} \mathcal{A}^\mathrm{et}_n(F) \xrightarrow{JL} \mathcal{A}^\mathrm{et}_m(D) \xrightarrow{{}_D\Pi^{-1}} P_n(F),
\end{equation}
determines a tamely ramified quadratic character ${}_D\nu_\xi$ of $E^\times$ for each admissible character $\xi$ of $E^\times$, depending only on the wild part of $\xi$, such that $(E/F,{}_D\nu_\xi \cdot \xi)$ is also admissible and
\begin{equation}\label{rectifier}
{}_D\nu(E/F, \xi) = (E/F, {}_D\nu_\xi \cdot \xi).
\end{equation}
We call this character ${}_D\nu_\xi$ the \emph{rectifier} of $\xi$ (for the Jacquet-Langlands correspondence). Using the First and Second Comparison Theorems of \cite{BH-JLC}, we can compute the values of ${}_D\nu_\xi$. To express these values, we need the knowledge of certain invariants of finite symplectic modules, which will be described in Section \ref{section Values of rectifiers}. Finally, with the expression of ${}_D\nu_\xi$, we see that we can describe the correspondence in Theorem \ref{ETJLC} explicitly, using (\ref{bijection of P and FPi}), (\ref{bijection of P and DPi}), and (\ref{composition is not identity}).

\subsection{Some subgroups}\label{section subgroups}

We recall certain subgroups of $G(F)$. Suppose that the jump data $\{ E_k, a_k| k=0,\dots,t\}$ are defined by the factorization (\ref{Howe factorization}) of an admissible pair $(E/F,\xi)$, or equivalently, of its associated 1-pair $(E_0/F,\Xi)$. We require that $(E_0,\mathfrak{A})$ satisfy the conditions in Section \ref{subsection An embedding condition}. We write $A_k$ the centralizer of $E_k$ in $A$ and $\mathfrak{A}_k = A_k\cap \mathfrak{A}$. We can then define $\mathfrak{P}_{\mathfrak{A}_k}$, $U_{\mathfrak{A}_k}$, $ U_{\mathfrak{A}_k}^x$, and $ U_{\mathfrak{A}_k}^{x+}$ for $x\in \mathbb{R}_{\geq0}$ analogously as in Section \ref{section hereditary order in CSA}. Following \cite[Definition 4.1]{Grab}, we construct the pro-$p$ subgroups
\begin{equation}\label{group H(xi) J(xi)}
\begin{split}
&H^1(\Xi,\mathfrak{A})= U_{\mathfrak{A}_{0}}^1 U_{\mathfrak{A}_{1}}^{(a_0e(\mathfrak{A}_1/\mathfrak{o}_E)/2)+} \cdots U_{\mathfrak{A}_t}^{(a_{t-1}e(\mathfrak{A}_t/\mathfrak{o}_E)/2)+} U_{\mathfrak{A}}^{(a_te(\mathfrak{A}/\mathfrak{o}_E)/2)+}\text{ and }
\\
&J^1(\Xi,\mathfrak{A})=U_{\mathfrak{A}_{0}}^1 U_{\mathfrak{A}_{1}}^{a_0e(\mathfrak{A}_1/\mathfrak{o}_E)/2} \cdots U_{\mathfrak{A}_t}^{a_{t-1}e(\mathfrak{A}_t/\mathfrak{o}_E)/2} U_{\mathfrak{A}}^{a_te(\mathfrak{A}/\mathfrak{o}_E)/2}.
\end{split}
\end{equation}
We also construct the subgroups
\begin{equation*}\label{group bold-J(xi}
  \begin{split}
&J(\Xi,\mathfrak{A})=U_{\mathfrak{A}_{0}}   U_{\mathfrak{A}_{1}}^{a_0e(\mathfrak{A}_1/\mathfrak{o}_E)/2} \cdots U_{\mathfrak{A}_t}^{a_{t-1}e(\mathfrak{A}_t/\mathfrak{o}_E)/2} U_{\mathfrak{A}}^{a_te(\mathfrak{A}/\mathfrak{o}_E)/2}
\text{ and}
\\
&\mathbf{J}(\Xi,\mathfrak{A})=E^\times J(\Xi,\mathfrak{A})=E_0^\times J(\Xi,\mathfrak{A}).
  \end{split}
\end{equation*}
We abbreviate these groups by $H^1,\, J^1,\, J$ and $\mathbf{J}$ if the admissible character $\xi$ is fixed. Notice that $H^1,\, J^1,J$ are compact subgroups of $G(F)$ and  $\mathbf{J}$ is a compact-mod-center subgroup of $G(F)$.

In \cite[Section 3.1]{Secherre1}, these subgroups are defined based on a simple stratum $[\mathfrak{A},-v_\mathfrak{A}(\beta),0,\beta]$, where $\beta$ is a suitable element in $E_0$, depending on $\Xi$ and such that
$$-v_\mathfrak{A}(\beta)=\text{ the }E\text{-level of }\Xi\circ N_{E/E_0}.$$
The group $H^1(\Xi,\mathfrak{A})$ is denoted by $H^1(\beta,\mathfrak{A})$ in \emph{loc. cit.} (and similarly for the other subgroups). This construction is an obvious generalization of \cite[Section 3.1]{BK} (see also \cite[Section 2.5]{BH-JLC} and the Comment therewithin).

\subsection{Simple characters}\label{section Simple Characters}

Given an admissible 1-pair $(E_0/F,\Xi)$ and a finite unramified extension $E/E_0$, we define $H^1(\Xi,\mathfrak{A})$ as in (\ref{group H(xi) J(xi)}). Using the idea of \cite[Section 3.2]{Moy} (see also \cite[Definitions 3.22, 3.45, Proposition 3.47]{Secherre1}), we attach to $(E_0/F,\Xi)$ a simple character $(H^1(\Xi,\mathfrak{A}),\theta_{\Xi,E})$ as follows. Suppose that $\Xi$ admits a factorization of the form (\ref{factorization of 1-pair}), with each $\xi_k\circ N_{E_0/E_k}$ a character of $U_{E_0}^1$. For each $\xi_k$ there is $c_k\in E_k\cap \mathfrak{p}_E^{-a_k}$ such that
\begin{equation}\label{multiplicative char and additive char c_k}
\xi_k\circ N_{E/E_k}(1+x)= \psi_F(\text{tr}_{E/F}(c_kx))\text{ for all }x\in \mathfrak{p}_E^{(a_k/2)+}.
\end{equation}
The element $c_k$ can be chosen mod $\mathfrak{p}_E^{-a_k/2}$. We denote the character on the right side of (\ref{multiplicative char and additive char c_k}) by $\psi_{c_k}$.

We define a character $\theta_{\Xi,E}$ of the subgroup $H^1(\Xi,\mathfrak{A})$ in (\ref{group H(xi) J(xi)}) by the following inductive procedure. We first take a character $\theta_n$ on the subgroup $U_{\mathfrak{A}_t}^{(a_{t-1}e(\mathfrak{A}_t/\mathfrak{o}_E)/2)+} U_{\mathfrak{A}}^{(a_te(\mathfrak{A}/\mathfrak{o}_E)/2)+}$ by
\begin{equation*}\label{simple char first step}
\begin{split}
&\xi_t\circ\mathrm{Nrd}_{A_t/E_t}\text{ on }U_{\mathfrak{A}_t}^{(a_{t-1}e(\mathfrak{A}_t/\mathfrak{o}_E)/2)+} \text{ and }
\\
&\psi_{c_t}\circ \mathrm{trd}_{\mathfrak{A}_t/E_t} \text{ on }U_{\mathfrak{A}}^{(a_te(\mathfrak{A}/\mathfrak{o}_E)/2)+}.
\end{split}
\end{equation*}
On the intersection $U_{\mathfrak{A}_t}^{(a_te(\mathfrak{A}_t/\mathfrak{o}_E)/2)+}$, the two characters agree (by Prop. \ref{char composing norm} and (\ref{multiplicative char and additive char c_k})). Hence $\theta_t$ is defined. Inductively, suppose $\theta_{k+1}$ is defined, we construct $\theta_k$ on the subgroup $$U_{\mathfrak{A}_k}^{(a_{k-1}e(\mathfrak{A}_{k}/\mathfrak{o}_E)/2)+} U_{\mathfrak{A}_{k+1}}^{(a_{k}e(\mathfrak{A}_{k+1}/\mathfrak{o}_E)/2)+}  \cdots U_{\mathfrak{A}}^{(a_te(\mathfrak{A}/\mathfrak{o}_E)/2)+}$$
by
\begin{equation}\label{simple char inductive step}
\begin{split}
&(\xi_k\circ\mathrm{Nrd}_{A_k/E_k})\cdots (\xi_{t+1}\circ\mathrm{Nrd}_{A/F})\text{ on }U_{\mathfrak{A}_k}^{(a_{k-1}e(\mathfrak{A}_{k}/\mathfrak{o}_E)/2)+}
\text{ and }
\\
&\theta_{k+1}\psi_{c_k}\circ \mathrm{trd}_{\mathfrak{A}_k/E_k}\text{ on } U_{\mathfrak{A}_{k+1}}^{(a_{k}e(\mathfrak{A}_{k+1}/\mathfrak{o}_E)/2)+}  \cdots U_{\mathfrak{A}}^{(a_te(\mathfrak{A}/\mathfrak{o}_E)/2)+}.
 \end{split}
\end{equation}
On the intersection $U_{\mathfrak{A}_k}^{(a_{k}e(\mathfrak{A}_k/\mathfrak{o}_E)/2)+}  $, we have $$\theta_{k+1}=(\xi_{k+1}\circ\mathrm{Nrd}_{A_{k+1}/E_{k+1}})\cdots (\xi_{t+1}\circ\mathrm{Nrd}_{A/F})$$
 and
 $$\psi_{c_k}\circ \mathrm{trd}_{\mathfrak{A}_k/E_k}=\xi_{k}\circ\mathrm{Nrd}_{A_{k}/E_{k}}.$$
  Hence the two characters in (\ref{simple char inductive step}) agree on the intersection and $\theta_k$ is defined. Finally, we define $\theta_{\Xi,E}=\theta_0$ on $H^1(\Xi,\mathfrak{A})$. By construction, this character $\theta_{\Xi,E}$ is normalized by $\mathfrak{K}_{\mathfrak{A}_0}$, hence it is a simple character by \cite[5.3 Definition.(i)]{Grab}.

To get back our character $\Xi$ from $\theta_{\Xi,E}$, just notice that $\theta_{\Xi,E}|_{U_E^1}$ factors through $N_{E/E_0}$, hence there is a unique character $\Xi$ of $U_{E_0}^1$ such that $\Xi\circ N_{E/E_0}=\theta_{\Xi,E}|_{U_E^1}$.

\subsection{Local constructions of attached pairs}\label{section local construction, attached pair}

We briefly summarize the construction in \cite[Sec. 4]{BH-JLC} of the bijection $  {}_D\Pi$ in (\ref{bijection of P and DPi}). We will distinguish between the `level-zero' case and the `positive-level' case (we refer to \cite[Sections 2.4 and 2.6]{BH-JLC} the definition of these cases ). As we remarked in the Introduction, it is indeed the inverse of $  {}_D\Pi$ that we are going to describe.

 In the level-zero case, the construction is similar to the one in the split case. Each level zero $\pi\in \mathcal{A}^\mathrm{et}_m(D)$ contains a representation $({\mathrm{GL}}_m(\mathfrak{o}_D),\lambda)$, called a maximal simple type of level zero, inflated from an irreducible cuspidal representation $({\mathrm{GL}}_m(\mathbf{k}_D),\bar{\lambda})$. This representation $\bar{\lambda}$ corresponds, via Green's parametrization \cite{Green}, to a $\mathbf{k}_E/\mathbf{k}_F$-regular character $\bar{\xi}$ of $\mathbf{k}_E^\times$, where $E/F$ is the unramified extension of degree $n$. We define the character $\xi$ of $E^\times$ such that $\xi|_{\mathfrak{o}_E^\times}$ is the inflation of $\bar{\xi}$ and $\xi|_{F^\times}$ is the central character of $\pi$. By \cite[4.2. Proposition]{BH-JLC}, the attached pair $(E/F,\xi)$ is admissible and the correspondence
\begin{equation}\label{level zero corresp}
  \mathcal{A}^\mathrm{et}_m(D)_{level-0}\rightarrow P_n(F)_{level-0},\,\pi\mapsto (E/F,\xi),
\end{equation}
is bijective. We can show that $\pi={\mathrm{cInd}}_{\mathbf{J}}^{G(F)}\Lambda$, where the condition $\delta(\pi)=n$ implies that $\mathbf{J}=F^\times {\mathrm{GL}}_m(\mathfrak{o}_D)$, and $(\mathbf{J},\Lambda)$ is defined by the conditions
\begin{equation}\label{condition-extension-level-0}
  \Lambda|_{{\mathrm{GL}}_m(\mathfrak{o}_D)}=\lambda\text{ and }\Lambda|_{F^\times}\text{ is a multiple of } \xi|_{F^\times}.
\end{equation}
The representation $(\mathbf{J},\Lambda)$ is called an extended maximal simple type of level zero.

In the positive level case, we first recall the construction of a extended maximal simple type in general. Suppose we have a simple character $(H^1,\theta)$. For example, we can construct a simple character $\theta=\theta_{\Xi,E}$ as in Section \ref{section Simple Characters} using an admissible pair $(E/F,\xi)$. We notice that the commutator subgroup $[J^1, J^1]$ lies in $H^1$ \cite[(3.1.15)]{BK}. By \cite[Th{\'e}or{\`e}me 3.52]{Secherre1}, the above simple character $ \theta$ induces an non-degenerate alternating bilinear form
\begin{equation}\label{alternating form assoc to simple character}
  h_\theta(x,y)=\theta([1+x,1+y])\text{, for all }1+x,1+y\in J^1,
\end{equation}
 on the $\mathbb{F}_p$-vector space
$${}_A\mathfrak{V}_\xi:=J^1/H^1.$$
  The classical theory of Heisenberg representation implies that there is a unique representation $\bar{\eta}$ of $J^1/\ker\theta$ containing the character $\theta$ of $H^1/\ker\theta$ as a central character. We then define ${\eta}$ as the inflation of $\bar{\eta}$ to $J^1$. By \cite[Th{\'e}or{\`e}me 2.28]{Secherre2}, there exists a unique irreducible representation $(J,\kappa)$, which is called a $\beta$-extension (or wide extension in \cite{BH-JLC}) of $\eta$ and satisfies certain conditions on its intertwining in $G(F)$ (see \cite[(2.5.5)]{BH-JLC}). We now choose a maximal simple type $( {\mathrm{GL}}_{m_0}(\mathfrak{o}_{D_0}),\sigma)$ of $A_0^\times=Z_{G(F)}(E_0^\times)$ of level zero and inflate it to a representation $({J},\sigma)$, since we know that ${J}={\mathrm{GL}}_{m_0}(\mathfrak{o}_{D_0})J^1$. We obtain a maximal simple type $(J,\lambda)$, where $\lambda=\kappa\otimes \sigma$. By \cite[Th{\'e}or{\`e}me 5.2]{Secherre3}, there exists an irreducible representation $\Lambda$ of $\mathbf{J}={E_0}^\times J^1$ (by the condition $\delta(\pi)=n$), extending $\lambda$ and whose compact-induction to $G(F)$ is irreducible and supercuspidal. The representation $(\mathbf{J},\Lambda)$ is called an extended maximal simple type. By \cite[Lemma 2 of Section 4.3]{BH-JLC}, we can fix a unique extended type containing $(J,\lambda)$ and satisfying the (non-canonical) conditions:
  $$\varpi_F\in \ker\Lambda\text{ and }\det\Lambda\text{ has a }p\text{-power order}.$$

Following \cite[Section 3 and 4]{BH-JLC}, we have to approach indirectly to describe the inverse of ${}_D\Pi$. Suppose that we are given a representation $\pi\in \mathcal{A}^\mathrm{et}_m(D)$ of positive level. By \cite[Th{\'e}or{\`e}me 5.21]{Secherre4}, it contains an extended maximal simple type $( \mathbf{J}, \Lambda)$ of the above form, such that $\Lambda|_{H^1}$ is a multiple of a simple character $(H^1,\theta)$. There is a unique character $\xi_w$ of $E^\times$, depending on $\theta|_{U^1_{\mathfrak{A}_0}}$, satisfying the conditions in (\ref{xi-1 conditions}). In particular, we have
$$\theta|_{U^1_{\mathfrak{A}_0}}=\Xi\circ \mathrm{Nrd}_{A_0/E_0}\text{, such that }\xi_w|_{U_{E}^1}=\Xi\circ N_{E/E_0}.$$
By the discussion of the previous paragraph, we can attach to $\xi_w$ an extended maximal simple type $( \mathbf{J}, \Lambda_w)$ such that $\Lambda\cong \Lambda_{-1}\otimes \Lambda_w$ for a uniquely determined extended maximal simple type $( \mathbf{J},\Lambda_{-1})$ of level zero. Attached to $( \mathbf{J},\Lambda_{-1})$ is a level zero character $\xi_{-1}$ of $E^\times$ admissible over $E_0$, as mentioned in the level zero case. Finally, by \cite[4.3. Proposition]{BH-JLC}, the attached pair $(E/F,\xi)$, where $\xi=\xi_{-1}\xi_w$, is admissible and independent of the various choices above. We call $(E/F,\xi)$ a pair attached to $\pi$.

The technical part is to show that the attaching map
 $$\mathcal{A}^\mathrm{et}_m(D)\rightarrow P_n(F),\, \pi\mapsto (E/F,\xi)$$
 is well-defined and injective. This is done in the Parametrization Theorem of \cite[Section 6]{BH-JLC}. The composition (\ref{composition is not identity}) is then injective (since the maps ${}_F\Pi$ and $JL$ are known to be bijective) and preserves the restriction of each character to the subgroup $F^\times U_E^1$, which is of finite index of $E^\times$. Therefore, the map in (\ref{composition is not identity}) and hence ${}_D\Pi$ in (\ref{bijection of P and DPi}), is bijective.

\subsection{Finite symplectic modules}

Since the group $\mathbf{J}$ normalizes the subgroups $H^1$, $J^1$, and the simple character $\theta$ of $H^1$, it acts on the finite quotient ${}_A\mathfrak{V}_\xi := J^1/H^1$. This quotient is denoted by ${}_M\mathfrak{V}_{\xi}$ in the split case $A=M$, which is considered in \cite{BH-ET3} and \cite{thesis}. Notice that the quotient is clearly a finite dimensional $\mathbb{F}_p$-vector space. The action of $\mathbf{J}$ induces a symplectic $\mathbb{F}_p\mathbf{J}$-module structure on this space with respect to the non-degenerate alternating form $h_\theta$ in (\ref{alternating form assoc to simple character}). We have a decomposition of
\begin{equation}\label{coarse decomp of V}
  {}_A\mathfrak{V}_\xi={}_A\mathfrak{V}_{\xi,0} \oplus \cdots \oplus {}_A\mathfrak{V}_{\xi,t},
\end{equation}
into $\mathbb{F}_p\mathbf{J}$-submodules, where
\begin{equation}\label{AVk}
\begin{split}
  {}_A\mathfrak{V}_{\xi,k} &=U_{\mathfrak{A}_{k+1}}^{a_ke(\mathfrak{A}_{k+1}/\mathfrak{o}_E)/2} / U_{\mathfrak{A}_{k}}^{a_{k}e(\mathfrak{A}_k/\mathfrak{o}_E)/2} U_{\mathfrak{A}_{k+1}}^{(a_ke(\mathfrak{A}_{k+1}/\mathfrak{o}_E)/2)+}
  \\
  &=\mathfrak{P}_{\mathfrak{A}_{k+1}}^{a_ke(\mathfrak{A}_{k+1}/\mathfrak{o}_E)/2} / \mathfrak{P}_{\mathfrak{A}_{k}}^{a_{k}e(\mathfrak{A}_k/\mathfrak{o}_E)/2}+ \mathfrak{P}_{\mathfrak{A}_{k+1}}^{(a_ke(\mathfrak{A}_{k+1}/\mathfrak{o}_E)/2)+}
\end{split}
\end{equation}
 for $k=0,\dots,t$. In this paper, We call this decomposition the \emph{coarse decomposition} of ${}_A\mathfrak{V}_\xi$. By \cite[Proposition 3.9]{Secherre1}, the decomposition (\ref{coarse decomp of V}) is orthogonal.

In the sequel, we will be interested in the adjoint action of $E^\times$ on ${}_A\mathfrak{V}_{\xi}$ restricted from that of $\mathbf{J}$, which factors through the finite group $\Psi_{E/F}:=E^\times/F^\times(E^\times\cap  J^1) $.

The following Proposition appears in {\cite[Proposition 5.6]{BH-JLC}}. We re-interpret its proof here.
\begin{prop}\label{totally ram AV equals MV}
If $E/F$ is totally ramified, then ${}_A\mathfrak{V}_{\xi}\cong {}_M\mathfrak{V}_{\xi}$.
\end{prop}
\proof
From the proofs of Propositions \ref{prop decomp-of-standard-module-div-alg} and \ref{decomp-of-standard-module-sim-alg} (which are purely algebra and do not require the knowledge of this section), we see that the totally ramified condition implies that $$\mathfrak{P}^j_{\mathfrak{M}_{k}}/\mathfrak{P}^{j+1}_{\mathfrak{M}_{k}}\cong \mathfrak{P}^j_{\mathfrak{A}_{k}}/\mathfrak{P}^{j+1}_{\mathfrak{A}_{k}}\cong \mathrm{Ind}_1^{\Psi_{E/E_k}}\mathbf{k}_{F}$$
as a $\mathbf{k}_{F}\Psi_{E/E_k}$-module, for all $j\in \mathbb{Z}$ and $k=1,\dots,t+1$. Moreover, each index $e(\mathfrak{A}_k/\mathfrak{o}_E)$ (appearing in the powers in (\ref{AVk})) is divisible by $f(E/F)$, which is equal to 1. Hence from the expression in (\ref{AVk}), we see that $ {}_A\mathfrak{V}_{\xi,k}$ and $ {}_M\mathfrak{V}_{\xi,k}$ are the same.
\qed

\subsection{Invariants of finite symplectic modules}\label{section Invariants of finite symplectic modules}

Let ${\boldsymbol{\Gamma}}$ be  a finite cyclic group whose order is not divisible by $p$. We call a finite $\mathbb{F}_p{\boldsymbol{\Gamma}}$-module ${V}$ \emph{symplectic} if there is a non-degenerate alternating form $\mathbf{h}:{V}\times {V}\rightarrow\mathbb{F}_p$ which is  ${\boldsymbol{\Gamma}}$-invariant, in the sense that
\begin{equation*}
\mathbf{h}({}^\gamma v_1,{}^\gamma v_2)=\mathbf{h}(v_1,v_2)\text{, for all }\gamma\in {\boldsymbol{\Gamma}},\,v_1,v_2\in {V}.
\end{equation*}
The simple module ${V}_{\lambda}$ corresponding to a character $\lambda\in \mathrm{Hom}({\boldsymbol{\Gamma}},\bar{\mathbb{F}}_p^\times)$ is the field  $\mathbb{F}_p[\lambda({\boldsymbol{\Gamma}})]$ generated over $\mathbb{F}_p$ by the image $\lambda({\boldsymbol{\Gamma}})$, with $\boldsymbol{\Gamma} $-action
$${}^\gamma v=\lambda(\gamma)v\text{, for all }\gamma\in \boldsymbol{\Gamma},\,v\in {V}_{\lambda}.$$
  Its $\mathbb{F}_p$-linear dual $ {V}_{\lambda}^*=\mathrm{Hom}({V}_{\lambda},{\mathbb{F}}_p)$ is  isomorphic to ${V}_{\lambda^{-1}}$ by the map
  $${V}_{\lambda^{-1}}\rightarrow  {V}_{\lambda}^*,\,v\mapsto (w\mapsto \mathrm{tr}_{\mathbb{F}_p[\lambda({\boldsymbol{\Gamma}})]/\mathbb{F}_p}(wv)),$$
   such that the canonical pairing $\left<\cdot,\cdot\right>:{V}_{\lambda}\times {V}_{\lambda^{-1}}\rightarrow \mathbb{F}_p$ is $\boldsymbol{\Gamma}$-invariant.

We recall some basic facts from  \cite[(8.2.3)]{BF} and \cite[Sec. 3, Prop. 4]{BH-ET3}.
\begin{prop} \label{hyper and aniso and properties}
\begin{enumerate}[(i)]
\item An indecomposable symplectic $\mathbb{F}_p{\boldsymbol{\Gamma}}$-module is isomorphic to either one of the following two kinds,
\begin{enumerate}
\item a \emph{hyperbolic} module of the form $\mathbf{V}_{\lambda}= {V}_{\lambda}\oplus {V}_{\lambda^{-1}}$ such that either $\lambda^2=1$ or ${V}_{\lambda}\ncong {V}_{\lambda^{-1}}$, with the alternating form
    $$\mathbf{h}_{\mathbf{V}_{\lambda}}((v_1,v_1^*),(v_2,v_2^*))=\left<v_1,v_2^*\right>-\left<v_2,v_1^*\right>\text{, for all }(v_1,v_1^*),(v_2,v_2^*)\in \mathbf{V}_{\lambda};$$ \label{def hyper and aniso hyper}
\item an \emph{anisotropic} module of the form ${V}_\lambda$ with $\lambda^2\neq1$ and ${V}_{\lambda}\cong {V}_{\lambda^{-1}}$. In this case, $ [\mathbb{F}_p[\lambda({\boldsymbol{\Gamma}})] : \mathbb{F}_p]$ is even and the alternating form $\mathbf{h}_{\mathbf{V}_{\lambda}}$ is defined, up to ${\boldsymbol{\Gamma}}$-isometry, by
    $$(v_1,v_2)\mapsto \mathrm{tr}_{\mathbb{F}_p[\lambda({\boldsymbol{\Gamma}})]/\mathbb{F}_p}(\alpha v_1\bar{v}_2)\text{, for all }v_1,v_2\in {V}_\lambda,$$
    where $v\mapsto \bar{v}$ is the $\mathbb{F}_p$-automorphism of $\mathbb{F}_p[\lambda({\boldsymbol{\Gamma}})]$ of order 2 and $\alpha\in\mathbb{F}_p[\lambda({\boldsymbol{\Gamma}})]^\times $ satisfies $\bar{\alpha}=-\alpha$. \label{def hyper and aniso aniso}
\end{enumerate} \label{def hyper and aniso}
\item   If ${V}_\lambda$ is anisotropic and $\mathbb{F}_p[\lambda(\Gamma)]_{\pm}$ denotes the subfield of $\mathbb{F}_p[\lambda(\Gamma)]$ such that $\mathbb{F}_p[\lambda(\Gamma)]/\mathbb{F}_p[\lambda(\Gamma)]_\pm$ is quadratic, then $\lambda({\boldsymbol{\Gamma}})$ is a subgroup of $\ker( N_{\mathbb{F}_p[\lambda({\boldsymbol{\Gamma}})] / \mathbb{F}_p[\lambda({\boldsymbol{\Gamma}})]_\pm})$ \label{even deg from def of aniso}
\item The ${\boldsymbol{\Gamma}}$-isometry class of a symplectic $\mathbb{F}_p{\boldsymbol{\Gamma}}$-module $({V},\mathbf{h})$ is determined by the underlying $\mathbb{F}_p{\boldsymbol{\Gamma}}$-module ${V}$. \label{symplectic determined by underlying module}
\end{enumerate}
\end{prop}

Part (\ref{symplectic determined by underlying module}) is particularly useful because, when we talk about invariants of  ${\boldsymbol{\Gamma}}$-isometry classes of symplectic $\mathbb{F}_p{\boldsymbol{\Gamma}}$-modules, we do not have to write down the alternating forms explicitly.

Given a finite symplectic $\mathbb{F}_p{\boldsymbol{\Gamma}}$-module $\mathfrak{V}$, we attach a sign $t^0_{\boldsymbol{\Gamma}}(\mathfrak{V})\in \{\pm1\}$ and a quadratic character $t^1_{\boldsymbol{\Gamma}}(\mathfrak{V})$ of ${\boldsymbol{\Gamma}}$. We also set $$t_{\boldsymbol{\Gamma}}(\mathfrak{V})=t^0_{\boldsymbol{\Gamma}}(\mathfrak{V})t^1_{\boldsymbol{\Gamma}}(\mathfrak{V})(\gamma),$$ where $\gamma$ is any generator of ${\boldsymbol{\Gamma}}$. We call these \emph{t-factors} of $\mathfrak{V}$.

We recall from \cite[Section 3]{BH-ET3} the definition the t-factors.
\begin{enumerate}[(i)]
\item If ${\boldsymbol{\Gamma}}$ acts on $\mathfrak{V}$ trivially, then $$t^0_{\boldsymbol{\Gamma}}(\mathfrak{V})=1\text{ and }t^1_{\boldsymbol{\Gamma}}(\mathfrak{V})\equiv1.$$
\item Let $\mathfrak{V}$ be an indecomposable symplectic $\mathbb{F}_p{\boldsymbol{\Gamma}}$-module.
\begin{enumerate}
\item If $\mathfrak{V}=\mathfrak{V}_\lambda\oplus \mathfrak{V}_{\lambda^{-1}}$ is hyperbolic, then $$t^0_{\boldsymbol{\Gamma}}(\mathfrak{V})=1\text{ and }t^1_{\boldsymbol{\Gamma}}(\mathfrak{V})=\mathrm{sgn}_{\lambda({\boldsymbol{\Gamma}})}(\mathfrak{V}_\lambda).$$
    Here $\mathrm{sgn}_{\lambda({\boldsymbol{\Gamma}})}(\mathfrak{V}_\lambda):{\boldsymbol{\Gamma}}\rightarrow\{\pm1\}$ is the character whose image $\gamma\mapsto \mathrm{sgn}_{\lambda(\gamma)}(\mathfrak{V}_\lambda)$ is the signature of the multiplication by $\lambda(\gamma)$ as a permutation of the set $\mathfrak{V}_\lambda$.
\item If $\mathfrak{V}=\mathfrak{V}_\lambda$ is anisotropic, then $$t^0_{\boldsymbol{\Gamma}}(\mathfrak{V})=-1\text{ and }t^1_{\boldsymbol{\Gamma}}(\mathfrak{V})(\gamma)=\left(\frac{\gamma}{\ker(N_{\mathbb{F}_p[\lambda({\boldsymbol{\Gamma}})]/\mathbb{F}_p[\lambda({\boldsymbol{\Gamma}})]_\pm})}\right)\text{ for any }\gamma\in {\boldsymbol{\Gamma}}.$$
    Here $\left(\frac{\cdot}{\cdot}\right)$ is the symbol defined as follows: for every finite cyclic group $H$, $$\left(\frac{x}{H}\right)=\begin{cases}
1 & \text{if }x\in H^2, \\
 -1 & \text{otherwise}.
\end{cases}$$
\end{enumerate}
\item If $\mathfrak{V}$ decomposes into an orthogonal sum $\mathfrak{V}_1\bot\cdots\bot \mathfrak{V}_t$ of indecomposable symplectic $\mathbb{F}_p{\boldsymbol{\Gamma}}$-modules, then $$t^i_{\boldsymbol{\Gamma}}(\mathfrak{V})=t^i_{\boldsymbol{\Gamma}}(\mathfrak{V}_1)\cdots t^i_{\boldsymbol{\Gamma}}(\mathfrak{V}_t)\text{ for }i=0,1.$$
\end{enumerate}
Notice that when $p=2$, the order of ${\boldsymbol{\Gamma}}$ is odd. In this case, $t^1_{\boldsymbol{\Gamma}}(\mathfrak{V})$ is always trivial, because all signature characters and symbols $\left(\frac{\cdot}{\cdot}\right)$ are trivial.

\subsection{Values of rectifiers}\label{section Values of rectifiers}

Given a tamely ramified extension $E/F$ and an $F$-admissible character $\xi$ of $E^\times$, let ${}_D\nu_\xi$ be the rectifier of $\xi$ defined in (\ref{rectifier}). To describe the values of ${}_D\nu_\xi$, we need to impose a condition on $\varpi_E$ defined in (\ref{e-th power of prime is also prime}):
\begin{equation}\label{addition condition on varpi_E}
  \varpi_E\in E_0,
\end{equation}
where $E_0$ be the first field appearing in the factorization (\ref{Howe factorization}) of $\xi$. This condition is the same as in the Second Comparison Theorem of \cite[Section 7]{BH-JLC}, where we further require that $\varpi_E^r\in F\text{ for some integer }r\text{ coprime to }p$. Indeed, from the assumption in (\ref{e-th power of prime is also prime}) this extra requirement is automatic in our situation. Under (\ref{addition condition on varpi_E}), the roots of unity ${{z}}_{E/F}$, ${{z}}_{\phi^i}$ (defined in Section \ref{section Galois groups}), and others related to $\varpi_E$ in later sections all depend on the first field $E_0$ in the jump data of $\xi$.

The values of the rectifier ${}_D\nu_\xi$ depends on the t-factors
$$t^1_{\boldsymbol{\mu}}({}_A\mathfrak{V}_{\xi}),\,t^1_{\boldsymbol{\mu}}({}_M\mathfrak{V}_{\xi}),\,t_\varpi({}_A\mathfrak{V}_{\xi})\text{ and } t_\varpi({}_M\mathfrak{V}_{\xi}),$$
where ${\boldsymbol{\mu}}={\boldsymbol{\mu}}_{E/F}$ and $\varpi=\varpi_{E/F}$ abbreviate the following subgroups of $\Psi_{E/F}=E^\times/F^\times U_E^1$,
\begin{equation}\label{finite cyclic subgp}
  {\boldsymbol{\mu}}:={\boldsymbol{\mu}}_E/{\boldsymbol{\mu}}_F\text{ and }\varpi:=\text{the subgroup generated by the image of }\varpi_E.
\end{equation}
By the First and Second Comparison Theorems of \cite{BH-JLC}, the rectifier ${}_D\nu_\xi$ has values
\begin{equation}\label{value of rectifier, sim alg}
 \begin{split}
   &{}_D{\nu}_{\xi}|_{{\boldsymbol{\mu}}_E}=t^1_{\boldsymbol{\mu}}({}_A\mathfrak{V}_{\xi})t^1_{\boldsymbol{\mu}}({}_M\mathfrak{V}_{\xi})
  \text{ and }
  \\
  &{}_D{\nu}_{\xi}(\varpi_E)=(-1)^{n-m+f_\varpi-m_\varpi}t_\varpi({}_A\mathfrak{V}_{\xi})t_\varpi({}_M\mathfrak{V}_{\xi}),
 \end{split}
\end{equation}
where $f_\varpi:=[E:F[\varpi_E]]=f(E/F[\varpi_E])$ and $m_\varpi=\gcd(m,f_\varpi)$.

In the case when $E/F$ is totally ramified, Proposition \ref{totally ram AV equals MV} implies that $${}_D{\nu}_{\xi}\text{ is unramified and } {}_D{\nu}_{\xi}(\varpi_E)=(-1)^{n-m},$$
as stated in \cite[5.3.Theorem]{BH-JLC}.

\section{Finite symplectic modules}\label{chapter Finite symplectic modules}

\subsection{Standard modules of central simple algebra}\label{section standard module}

Let $\mathfrak{A}$ be the hereditary $E$-pure order in $A$, as discussed in Section \ref{section hereditary order in CSA}. The isomorphism (\ref{A mod PA}) implies that $\mathfrak{P_A}^j/\mathfrak{P_A}^{j+1}\cong \mathrm{Mat}_s(\mathbf{k}_D)^r$ for all $j\in\mathbb{Z}$, where $s=\gcd(f,m)$ and $r=e/\gcd(d,e)$. We denote this quotient by $(\mathrm{Mat}_s(\mathbf{k}_D)^r)_j$ when we want to emphases the index $j$. Notice that as $\mathbf{k}_F\Psi_{E/F}$-modules, all  $(\mathrm{Mat}_s(\mathbf{k}_D)^r)_j$, for $j$ ranges over all $\mathbb{Z}$, are isomorphic to each other.

 When $\mathfrak{A}=\mathfrak{M}$ and $j=0$, we know that $\mathfrak{U}_\mathfrak{M}:=\mathfrak{M} / \mathfrak{P}_{\mathfrak{M}}\cong\mathrm{Mat}_n(\mathbf{k}_F)$ admits a root-space decomposition
\begin{equation*}
\mathfrak{U}_\mathfrak{M} \cong \mathfrak{U}_{0} \bigoplus_{[\lambda]\in\Gamma_F \backslash \Phi}\mathfrak{U}_{[\lambda]},
\end{equation*}
where $\mathfrak{U}_{0}\cong \mathfrak{o}_E/\mathfrak{p}_E$ on which $ \Psi_{E/F}$ acts trivially, and $\mathfrak{U}_{[\lambda]}$ is the $\mathbf{k}_F$-subspace on which $ \Psi_{E/F}$ acts by the character $\lambda$. Note that the equivalence class of the $\mathbf{k}_F\Psi_{E/F}$-module $\mathfrak{U}_{[\lambda]}$ depends only on the $\Gamma_F$-orbit of $\lambda$.

For future computation,  we rewrite the above decomposition as
\begin{equation}\label{decomp-of-standard-module-split-case}
\mathfrak{U}_\mathfrak{M}\cong  \bigoplus_{[g]\in\Gamma_E \backslash \Gamma_F / \Gamma_E}\mathfrak{U}_{[g]}
\end{equation}
using the identification in Proposition \ref{orbit of roots as double coset}. Here $\mathfrak{U}_{[g]}\cong \mathbf{k}_{E({}^gE)}$ as a $\mathbf{k}_F$-vector space for each $[g]\in\Gamma_E \backslash \Gamma_F / \Gamma_E $, and the $\Psi_{E/F}$-action on each $ v\in \mathfrak{U}_{[g]}$ is given as follows: if $[g]=[\sigma^i\phi^j]$, then
\begin{equation}\label{explicit action of psi}
  {}^{{z}} v= ({{z}}^{q^j-1})^{-1}v\text{ for all }{{z}}\in {\boldsymbol{\mu}}_E\text{ and }{}^{\varpi_E}v= ({{z}}^i_e {{z}}_{\phi^j})^{-1}v,
\end{equation}
where ${{z}}_e$ and $ {{z}}_{\phi^j}$ are defined in Section \ref{section Galois groups}.

For general inner form $G$, we first consider a simple case when  $A$ is a division algebra $D$. We write $(\mathbf{k}_D)_j:=\mathfrak{P}_D^j/\mathfrak{P}_D^{j+1}$ for each $j\in \mathbb{Z}$.
\begin{prop}\label{prop decomp-of-standard-module-div-alg}
For each $j\in\mathbb{Z}$, the $\mathbf{k}_F\Psi_{E/F}$-module $ (\mathbf{k}_D)_j$
  is isomorphic to
\begin{equation}\label{decomp-of-standard-module-div-alg}
\bigoplus_{[\sigma^i\phi^{hj}]\in\Gamma_E \backslash \Gamma_F / \Gamma_E} \mathfrak{U}_{[\sigma^i\phi^{hj}]}.
\end{equation}
\end{prop}
\begin{proof}
Recall that, if we denote by $K$ the maximal unramified extension in $D$ (of degree $n$ over $F$), then
$$\mathfrak{P}^i_D=\cdots\oplus\mathfrak{p}_K\varpi^i_D\oplus \mathfrak{o}_K\varpi^{i+1}_D\oplus\cdots.$$
  Hence we can use ${{z}}\varpi_D^i$, with ${{z}}\in {\boldsymbol{\mu}}_K\cup\{0\}$, as a representative in $\mathfrak{P}^i_D$ of an element in $(\mathbf{k}_D)_j$. We regard $(\mathbf{k}_D)_j$ as a $\mathbf{k}_E$-vector space of dimension $e$
  such that ${{z}}\in {\boldsymbol{\mu}}_E$ acts on each piece $(\mathbf{k}_E)_j$ by the character $\left[\begin{smallmatrix}
  1\\ \phi^{hj}
\end{smallmatrix}\right]({{z}})={{z}}^{1-q^{hj}}$ as in (\ref{ad-zeta-on-varpi-D-i}). Therefore, a Frobenius reciprocity argument (which is still valid when $p$ does not divide $\#\Psi_{E/F}$) implies that
$$(\mathbf{k}_D)_j\cong {\mathrm{Ind}}_{{\boldsymbol{\mu}}_E/{\boldsymbol{\mu}}_F}^{\Psi_{E/F}}(\mathbf{k}_E)_j;$$
 more precisely, the action of $\varpi_E$ has eigenvalues $({{z}}_{e}^i{{z}}_{\phi^{hj}})^{-1}$, $i=0,\dots,e-1$, where we recall that ${{z}}_{\phi^{hj}}$ is an $e$th root of ${{z}}^{q^{hj}-1}_{E/F}$. Hence we have the decomposition (\ref{decomp-of-standard-module-div-alg}) with the $\Psi_{E/F}$-action on each component
$ \mathfrak{U}_{[\sigma^i\phi^{hj}]}$ as in (\ref{explicit action of psi}) for each fixed $j$.
\end{proof}

For the general $A$, if we write $(\mathrm{Mat}_s(\mathbf{k}_D)^r)_{j'}:=\mathfrak{P_A}^{j'}/\mathfrak{P_A}^{{j'}+1} $ for each $j'\in \mathbb{Z}$, then we have the following result.
\begin{prop}\label{decomp-of-standard-module-sim-alg}
 For each $j'\in \mathbb{Z}$, we have a decomposition
  $$(\mathrm{Mat}_s(\mathbf{k}_D)^r)_{j'}\cong \bigoplus_{
  \begin{smallmatrix}
    [\sigma^i\phi^{j}]\in \Gamma_E\backslash \Gamma_F/\Gamma_E
    \\
     j\equiv hj' \mod f/s
  \end{smallmatrix}
  } \mathfrak{U}_{[\sigma^i\phi^{j}]}.$$
  as a $\Psi_{E/F}$-module. (Recall that $f/s=d/\gcd(d,e)=e(\mathfrak{A}/\mathfrak{o}_E)$.)
  \end{prop}

\proof
Denote by $E_D=E\cap D$ the maximal subfield contained in both $E$ and $D$, then $[\mathbf{k}_{E_D}:\mathbf{k}_{F}]=f/s$ and $[\mathbf{k}_{D}:\mathbf{k}_{E_D}]=\gcd(d,e)=e/r$.
By Proposition \ref{prop decomp-of-standard-module-div-alg}, we know that ${{z}}\in {\boldsymbol{\mu}}_{E_D}$ acts on $(\mathrm{Mat}_s(\mathbf{k}_D)^r)_{j'}$ as a sum of $(\mathbf{k}_{E_D})_{hj'}$, i.e., ${{z}}$ acts by the character $\left[\begin{smallmatrix}
  1\\ \bar{\phi}^{hj'}
\end{smallmatrix}\right]({{z}})={{z}}^{1-q^{hj'}}$, where $\bar{\phi}$ is the image of $\phi$ under the natural projection
\begin{equation}\label{proj of finite Gal gp}
\Gamma_{\mathbf{k}_E/\mathbf{k}_F}\rightarrow \Gamma_{\mathbf{k}_{E_D}/\mathbf{k}_F}.
\end{equation}
(Note that the arguments above concerning Proposition \ref{prop decomp-of-standard-module-div-alg} still valid even though $E_D$ may not be a maximal subfield of $D$.) We now consider the $\mathbf{k}_{E_D}$-embeddings (where all choices are conjugate to each other)
$$\mathbf{k}_{E_D}\xrightarrow{}\mathbf{k}_E\xrightarrow{}\mathrm{Mat}_s(\mathbf{k}_{E_D})\xrightarrow{}\mathrm{Mat}_s(\mathbf{k}_D),$$
Notice that $\mathbf{k}_E$ is a maximal subfield of $\mathrm{Mat}_s(\mathbf{k}_{E_D})$. By the ``twisted group ring decomposition'', we know that ${{z}}\in {\boldsymbol{\mu}}_{E}$ acts on $\mathrm{Mat}_s(\mathbf{k}_{E_D})$ as a sum of $((\mathbf{k}_E)_{hj'})^{s}$, i.e., ${{z}}$ acts on each of the $s$ summands of $\mathbf{k}_E$ by the character $\left[\begin{smallmatrix}
  1\\ {\phi}^{j}
\end{smallmatrix}\right]({{z}})={{z}}^{1-q^{j}}$, for ${\phi}^{j}$ ranges over the $s$ pre-images of $\bar{\phi}^{hj'}$ under the natural projection (\ref{proj of finite Gal gp}). We denote this ${\boldsymbol{\mu}}_E$-module by $\mathrm{Mat}_s(\mathbf{k}_{E_D})_{j'}$. Finally, since the relative degree of $(\mathrm{Mat}_s(\mathbf{k}_D))^r$ over $\mathrm{Mat}_s(\mathbf{k}_{E_D})$ is $e=[\Psi_{E/F}:{\boldsymbol{\mu}}_E/{\boldsymbol{\mu}}_F]$, a Frobenius reciprocity argument (which is still valid when $p$ does not divide $\#\Psi_{E/F}$) implies that
$$(\mathrm{Mat}_s(\mathbf{k}_D)^r)_{j'}\cong {\mathrm{Ind}}_{{\boldsymbol{\mu}}_E/{\boldsymbol{\mu}}_F}^{\Psi_{E/F}}\mathrm{Mat}_s(\mathbf{k}_{E_D})_{j'};$$
Therefore, we have obtained the desired decomposition and proved the proposition.
\qed

The following Corollary is a direct consequence of Proposition \ref{decomp-of-standard-module-sim-alg}.
  \begin{cor}
    The graded algebra
  \begin{equation*}\label{sum of graded algebra}
 \mathfrak{U}_\mathfrak{A} :=\bigoplus_{j'=0}^{f/s-1}(\mathrm{Mat}_s(\mathbf{k}_D)^r)_{j'},
  \end{equation*}
  is isomorphic to $ \mathfrak{U}_\mathfrak{M} $  as a $\Psi_{E/F}$-module.
  \qed\end{cor}

We provide some notations for later use. We write
$$\mathfrak{U}_{\mathrm{sym}} :=\bigoplus_{[g]\in ( \Gamma_{E}\backslash \Gamma_{F}/\Gamma_{E})_{\mathrm{sym}}}\mathfrak{U}_{[g]}$$
and also $\mathfrak{U}_{{\mathrm{sym-ram}}}$ and $\mathfrak{U}_{{\mathrm{sym-unram}}}$ analogously. Given intermediate extensions $F\subseteq K\subseteq L\subseteq E$, we write
$$\mathfrak{U}_{K/L} :=\bigoplus_{[g]\in  \Gamma_{E}\backslash (\Gamma_{L}-\Gamma_{K})/\Gamma_{E}}\mathfrak{U}_{[g]}$$
We also define the \emph{symmetric module} associated to $\mathfrak{U}_{[g]}$ (or $\mathfrak{U}_{[g^{-1}]}$) by
\begin{equation}\label{def sym mod}
\boldsymbol{\mathfrak{U}}_{[g]}:=\begin{cases}
  \mathfrak{U}_{[g]}\oplus \mathfrak{U}_{[g^{-1}]}&\text{ if }[g]\text{ is asymmetric, }
  \\
  \mathfrak{U}_{[g]}&\text{ if }[g]\text{ is symmetric, }
\end{cases}
\end{equation}
and call
\begin{equation}\label{complete decomposition}
\mathfrak{U}_{\mathfrak{A}} \cong \mathfrak{U}_{0} \bigoplus_{[g]\in (\Gamma_E\backslash \Gamma_F / \Gamma_E)_{\mathrm{sym}}\sqcup (\Gamma_E\backslash \Gamma_F / \Gamma_E)_{\mathrm{asym}/\pm}}\boldsymbol{\mathfrak{U}}_{[g]}
\end{equation}
the \emph{complete symmetric decomposition} of $\mathfrak{U}_{\mathfrak{A}}$. If $\mathfrak{V}$ is a submodule of $\mathfrak{U}_{\mathfrak{A}}$, we also use the same convention to denote its submodules, for example, $\mathfrak{V}_{K/L}=\mathfrak{U}_{K/L}\cap \mathfrak{V}$ and $\boldsymbol{\mathfrak{V}}_{[g]}=\boldsymbol{\mathfrak{U}}_{[g]}\cap \mathfrak{V}$, and also call
$$
\mathfrak{V} \cong (\mathfrak{U}_{0}\cap \mathfrak{V} )\bigoplus_{[g]\in (\Gamma_E\backslash \Gamma_F / \Gamma_E)_{\mathrm{sym}}\sqcup (\Gamma_E\backslash \Gamma_F / \Gamma_E)_{\mathrm{asym}/\pm}}\boldsymbol{\mathfrak{V}}_{[g]}
$$
the {complete symmetric decomposition} of $\mathfrak{V}$

\subsection{Complete decomposition of finite symplectic modules}\label{section Complete decomposition of finite symplectic modules}

We are interested in the adjoint action of $E^\times$ on ${}_A\mathfrak{V}_{\xi}$ restricted from that of $\mathbf{J}$, which factors through the finite group $E^\times/F^\times(E^\times\cap  J^1) \cong \Psi_{E/F}$. We also know that this action preserves the symplectic structure $h_\theta$ (\ref{alternating form assoc to simple character}) on ${}_A\mathfrak{V}_{\xi}$. Hence ${}_A\mathfrak{V}_{\xi}$ is moreover a finite symplectic $\mathbb{F}_p\Gamma$-module for each cyclic subgroup $\Gamma$ of $\Psi_{E/F}$. We denote the $\mathfrak{U}_{[g]}$-isotypic component in ${}_A\mathfrak{V}_{\xi}$ by ${}_A\mathfrak{V}_{\xi,[g]}$, and obtain the decompositions
\begin{equation}\label{complete decomposition-V}
  {}_A\mathfrak{V}_{\xi}=\bigoplus_{[g]\in (\Gamma_E\backslash \Gamma_F / \Gamma_E)'}{}_A\mathfrak{V}_{\xi,[g]}=\bigoplus_{[g]\in (\Gamma_E\backslash \Gamma_F / \Gamma_E)_{\mathrm{sym}}\sqcup (\Gamma_E\backslash \Gamma_F / \Gamma_E)_{\mathrm{asym}/\pm}}{}_A\boldsymbol{\mathfrak{V}}_{\xi,[g]}
\end{equation}
 inherited from (\ref{decomp-of-standard-module-split-case}) and (\ref{complete decomposition}) respectively. These decomposition are finer than the one in (\ref{coarse decomp of V}). Indeed, it is easy to see that
 \begin{equation*}
  {}_A\mathfrak{V}_{\xi,k}:={}_A\mathfrak{V}_{\xi,E_k/E_{k+1}}=\bigoplus_{[g]\in \Gamma_E\backslash (\Gamma_{E_{k+1}}-\Gamma_{E_k}) / \Gamma_E}{}_A\mathfrak{V}_{\xi,[g]}
\end{equation*}
for $k=0,\dots,t$.

\begin{prop}\label{complete decomp orthogonal}
The complete symmetric decomposition of ${}_A\mathfrak{V}_{\xi}$ is orthogonal with respect to the alternating form $h_\theta$.
\end{prop}
\begin{proof}
Since we know that the $\Psi_{E/F}$-components of ${}_A\mathfrak{V}_{\xi}$ consist of those in the standard module $\mathfrak{U}_\mathfrak{A}$, which is isomorphic to the standard one $\mathfrak{U}_\mathfrak{M}$ in the split case, the proof of the assertion is just analogous to the one in the split case \cite[(5.10)]{thesis}, based on the argument of \cite[(8.2.3),(8.2.4)]{BF}.
\end{proof}

We would like to describe the isotypic component appearing in the complete decomposition (\ref{complete decomposition-V}) of ${}_A\mathfrak{V}_{\xi}$. We first write $e({\mathfrak{A}}/{\mathfrak{A}_{k+1}}):= e({\mathfrak{A}}/{\mathfrak{o}_{E}})/e({\mathfrak{A}_{k+1}}/{\mathfrak{o}_{E}})$.

\begin{prop}\label{non-zero quotient}
 The quotient $\mathfrak{P}_\mathfrak{A}^j\cap A_{k+1}/\mathfrak{P}_\mathfrak{A}^{j+}\cap A_{k+1}$ is non-trivial if and only if $j\in e({\mathfrak{A}}/{\mathfrak{A}_{k+1}})\mathbb{Z}$.
\end{prop}

\begin{proof}
 Since $\mathfrak{P}_\mathfrak{A}^j\cap A_{k+1}=\mathfrak{P}_{\mathfrak{A}_{k+1}}^{ j/e({\mathfrak{A}}/{\mathfrak{A}_{k+1}})}$
for all $j\in \mathbb{Z}$, the assertion follows directly.
\end{proof}

We now specify $j=j_k=e({\mathfrak{A}}/{\mathfrak{o}_{E}})a_k/2$ for some integer $a_k$, and so $\mathfrak{P}_\mathfrak{A}^{j_k}\cap A_{k+1}=\mathfrak{P}_{\mathfrak{A}_{k+1}}^{a_ke({\mathfrak{A}}_{k+1}/{\mathfrak{o}_{E}})/2}$, such that the index on the right side is the one appearing in the group $J^1$ (\ref{group H(xi) J(xi)}). The condition in Proposition \ref{non-zero quotient} is satisfied if and only if $a_k$ is even or $e({\mathfrak{A}_{k+1}}/{\mathfrak{o}_{E}})$ is even, in which case
\begin{equation*}
\begin{split}
{}_A\mathfrak{V}_{\xi,k}
&\cong\frac{\mathfrak{P}_{\mathfrak{A}_{k+1}}^{a_ke(\mathfrak{A}_{k+1}/\mathfrak{o}_E)/2}}{\mathfrak{P}^{{a_ke(\mathfrak{A}_{k+1}/\mathfrak{o}_E)/2}+}_{\mathfrak{A}_{k+1}} +\mathfrak{P}^{{a_ke(\mathfrak{A}_{k}/\mathfrak{o}_E)/2}}_{ \mathfrak{A}_{k}}}
\\
&\cong\begin{cases}
  \mathrm{Mat}_{s_{k+1}}(\mathbf{k}_{D_{{k+1}}})^{r_{k+1}}/\mathrm{Mat}_{s_{k}}(\mathbf{k}_{D_{{k}}})^{r_{k}} & \text{when }a_ke(\mathfrak{A}_{k}/\mathfrak{o}_E)/2\in\mathbb{Z},
   \\
  \mathrm{Mat}_{s_{k+1}}(\mathbf{k}_{D_{{k+1}}})^{r_{k+1}} &\text{otherwise},
\end{cases}
\end{split}
\end{equation*}
where $r_k$ and $s_k$ are the invariants of $\mathfrak{A}_k$ analogous to $r$ and $s$ of $\mathfrak{A}$.

To summarize,
\begin{equation}\label{sim alg jump module}
{}_A\mathfrak{V}_{\xi,k}\cong\begin{cases}
0& \text{when }a_k\text{ is odd and }e({\mathfrak{A}}_{k+1}/{\mathfrak{o}_{E}})\text{ is odd,}
\\
 \mathrm{Mat}_{s_{k+1}}(\mathbf{k}_{D_{{k+1}}})^{r_{k+1}}  &\text{when }a_k\text{ is odd, }e({\mathfrak{A}}_{k}/{\mathfrak{o}_{E}})\text{ is odd, and }e({\mathfrak{A}}_{k+1}/{\mathfrak{o}_{E}})\text{ is even},
\\
\mathrm{Mat}_{s_{k+1}}(\mathbf{k}_{D_{{k+1}}})^{r_{k+1}}/\mathrm{Mat}_{s_{k}}(\mathbf{k}_{D_{{k}}})^{r_{k}} &\text{when } a_k\text{ is even or }e({\mathfrak{A}}_{k}/{\mathfrak{o}_{E}})\text{ is even}.
\end{cases}
\end{equation}

The action of $\Psi_{E/F}$ on $\mathfrak{P}_\mathfrak{A}^{j_k} \cap A_{k+1}$ is given by $\sigma^i\phi^{jh}\in \Gamma_{E_{k+1}}$, where $j$ has image $h{j_k}$ in the natural projection $\mathbb{Z}/f\mathbb{Z}\rightarrow \mathbb{Z}/e({\mathfrak{A}}/{\mathfrak{o}_{E}})\mathbb{Z}$. Therefore, directly from (\ref{sim alg jump module}), we have the following decompositions.
\begin{prop}\label{describe complete decomp}
  The complete decomposition of the component ${}_A\mathfrak{V}_{\xi,k}$ is given by
  \begin{enumerate}[(i)]
   \item When  $a_k$ is odd and $e({\mathfrak{A}}_{k+1}/{\mathfrak{o}_{E}})$ is odd, then ${}_A\mathfrak{V}_{\xi,k}$ is trivial.
   \item When $a_k$ is odd, $e({\mathfrak{A}}_{k}/{\mathfrak{o}_{E}})$ is odd, and $e({\mathfrak{A}}_{k+1}/{\mathfrak{o}_{E}})$ is even, then
       $${}_A\mathfrak{V}_{\xi,k}\cong
       \bigoplus_{\begin{smallmatrix}
         [g]=[\sigma^i\phi^{j}]\in \Gamma_E\backslash \Gamma_{E_{k+1}} / \Gamma_E
         \\
         j\equiv hj_k \mod e({\mathfrak{A}}_{}/{\mathfrak{o}_{E}})
       \end{smallmatrix}}\mathfrak{U}_{[g]}.$$ \label{describe complete decomp odd-even}
       \item When $a_k$ is even or  $e({\mathfrak{A}}_{k}/{\mathfrak{o}_{E}})$ is even, then
       $${}_A\mathfrak{V}_{\xi,k}\cong
       \bigoplus_{\begin{smallmatrix}
         [g]=[\sigma^i\phi^{j}]\in \Gamma_E\backslash( \Gamma_{E_{k+1}}- \Gamma_{E_{k}}) / \Gamma_E
         \\
         j\equiv hj_k \mod e({\mathfrak{A}}_{}/{\mathfrak{o}_{E}})
       \end{smallmatrix}}\mathfrak{U}_{[g]}.$$
  \end{enumerate}
\qed\end{prop}

\subsection{Some properties of parities of jumps}

Let $R$ be the index when $f(E/E_R)$ is odd and $f(E/E_{R+1})$ is even.
\begin{lem}\label{sym-unram-part of WE(R+1)}
 We have $( \Gamma_{E_{R+1}}/\Gamma_E)_{{\mathrm{sym-unram}}}=( ( \Gamma_{E_{R+1}}-\Gamma_{E_{R}})/\Gamma_E)_{{\mathrm{sym-unram}}}$.
\end{lem}
\begin{proof}
Recall Proposition \ref{properties of symmetric [g]} that every symmetric unramified $[g]$ are of the form $[\sigma^i\phi^{f/2}]$, so there is no coset of the form $\sigma^i\phi^{f/2}$ belonging to $\Gamma_{E_{R}}$.
\end{proof}

Let $Q$ be the index when $e({\mathfrak{A}}_{Q}/{\mathfrak{o}_{E}})$ is odd and $e({\mathfrak{A}}_{Q+1}/{\mathfrak{o}_{E}})$ is even.
\begin{lem}\label{2power-condition-QR}
Suppose that $f$ is even. We always have $R\leq Q$. If moreover $m$ is odd, then $Q=R$.
  \end{lem}
\begin{proof}
   We know that $e({\mathfrak{A}}_{k}/{\mathfrak{o}_{E}})$ divides $f(E/E_k)$, so that if $Q<R$, then the even number $e({\mathfrak{A}}_{Q+1}/{\mathfrak{o}_{E}})$ divides $f(E/E_{Q+1})$, which divides the odd number $f(E/E_R)$. This is a contradiction. Hence $R\leq Q$. When $R\lneq Q$, then
   \begin{equation}\label{e-odd-f-even}
    e({\mathfrak{A}}_{R+1}/{\mathfrak{o}_{E}})\text{ is odd and }f(E/E_{R+1})\text{ is even. }
   \end{equation}
   Since
  \begin{equation}\label{fraction e}
    e(\mathfrak{A}_{R+1}/\mathfrak{o}_E)=f_{E_{R+1}}/s_{E_{R+1}}=\frac{f(E/E_{R+1})}{\gcd(f(E/E_{R+1}),\gcd(m,n(E/E_{R+1})))},
  \end{equation}
  the statement (\ref{e-odd-f-even}) is equivalent to saying that
   \begin{equation}\label{2-power-equal}
     \text{the 2-powers of the numerator and the denominator on the right side of (\ref{fraction e}) are equal. }
   \end{equation}
   This power is greater than 0. Hence (\ref{2-power-equal}) is equivalent to that
\begin{equation}\label{2-power-non-neg}
(\text{the 2-power of }m)\geq (\text{the 2-power of }f(E/E_{R+1}))\gneq 0.
  \end{equation}
  If $m$ is odd, then (\ref{2-power-non-neg}) is a contradiction.
  \end{proof}

\subsection{Symmetric submodules}
We write ${}_A\mathfrak{V}_{\xi,\mathrm{sym}}={}_A\mathfrak{V}_{\xi}\cap \mathfrak{U}_{\mathrm{sym}}$ and ${}_A\mathfrak{V}_{\xi,{\mathrm{sym-ram}}}$ and ${}_A\mathfrak{V}_{\xi,{\mathrm{sym-unram}}}$ analogously.

\subsubsection{Case when $f$ is odd}

From Proposition \ref{totally ram AV equals MV}, we always have
\begin{equation}\label{sym submod f odd}
 {}_A\mathfrak{V}_{\xi,\mathrm{sym}}= {}_A\mathfrak{V}_{\xi,{\mathrm{sym-ram}}}\cong {}_M\mathfrak{V}_{\xi,{\mathrm{sym-ram}}}={}_M\mathfrak{V}_{\xi,\mathrm{sym}}.
\end{equation}

\subsubsection{Case when $f$ is even}
Notice that the natural projection $\mathbb{Z}/f\mathbb{Z}\rightarrow \mathbb{Z}/e({\mathfrak{A}}/{\mathfrak{o}_{E}})\mathbb{Z}$ maps
\begin{equation}\label{f/2-natural-proj}
f/2\mapsto
\begin{cases}
  0 & \text{ if }e({\mathfrak{A}}/{\mathfrak{o}_{E}})\text{ divides }f/2,
  \\
  e({\mathfrak{A}}/{\mathfrak{o}_{E}})/2 \neq 0 & \text{ otherwise }
  \end{cases}
\end{equation}
The condition that $e({\mathfrak{A}}/{\mathfrak{o}_{E}})$ divides $f/2$ is equivalent to that $s$ is even. When $f$ is even, then $s=\gcd(f,m)$ is even if and only if $m$ is even. We hence separate the cases according to the parity of $m$.

\subsubsection{Case when both $f$ and $m$ are even}

In this case, $f/2$ is mapped to 0 by $\mathbb{Z}/f\mathbb{Z}\rightarrow \mathbb{Z}/e({\mathfrak{A}}/{\mathfrak{o}_{E}})\mathbb{Z}$. We separate the cases according to the parity of the jump $a_k$. When $a_k$ is odd, neither 0 or $f/2$ is mapped to $hj_k\neq 0\in \mathbb{Z}/e({\mathfrak{A}}/{\mathfrak{o}_{E}})\mathbb{Z}$, and so ${}_A\mathfrak{V}_{\xi,k}$ is trivial. When $a_k$ is even, both $0$ and $f/2$ are mapped to $hj_k=0$ by (\ref{f/2-natural-proj}), and so ${}_A\mathfrak{V}_{\xi,k,\mathrm{sym}}\cong \mathfrak{U}_{k,\mathrm{sym}}$.

 We also recall that
 \begin{equation*}
   {}_M\mathfrak{V}_{\xi,k}\cong \begin{cases}
     0 &\text{if }a_k\text{ is odd,}
     \\
     \mathfrak{U}_{k,\mathrm{sym}}&\text{if }a_k\text{ is even.}
   \end{cases}
 \end{equation*}

Whatever the parity of $a_k$ is, we always have ${}_A\mathfrak{V}_{\xi,\mathrm{sym}}\cong {}_M\mathfrak{V}_{\xi,\mathrm{sym}}$.

\subsubsection{Case when $f$ is even and $m$ is odd}

In this case, notice that $e({\mathfrak{A}}/{\mathfrak{o}_{E}})$ must be even, and
  \begin{equation*}
j_k=e({\mathfrak{A}}/{\mathfrak{o}_{E}})a_k/2\equiv\begin{cases}
  0
  \\
  e({\mathfrak{A}}/{\mathfrak{o}_{E}})/2
\end{cases}
\mod e({\mathfrak{A}}/{\mathfrak{o}_{E}})
\begin{cases}
   \text{ if }a_k\text{ is even,}
  \\
   \text{ if }a_k\text{ is odd.}
\end{cases}
  \end{equation*}

  Therefore,
    \begin{equation*}\label{f/e(A/OE)-odd-type}
    {}_A\mathfrak{V}_{\xi,k,\mathrm{sym}}= \begin{cases}
{}_A\mathfrak{V}_{\xi,k,{\mathrm{sym-ram}}} & \text{ if }a_k\text{ is even,}
  \\
{}_A\mathfrak{V}_{\xi,k,{\mathrm{sym-unram}}} & \text{ if }a_k\text{ is odd.}
\end{cases}
  \end{equation*}
Using Proposition \ref{describe complete decomp}, we find that when $a_k$ is even,
$$ {}_A\mathfrak{V}_{\xi,k,\mathrm{sym}}=\mathfrak{U}_{E_k/E_{k+1},{\mathrm{sym-ram}}}=\bigoplus_{[\sigma^i]\in(\Gamma_E\backslash (\Gamma_{E_{R+1}}-\Gamma_{E_{R}})/\Gamma_E)_{\mathrm{sym}}} \mathfrak{U}_{[\sigma^i]},$$
and when $a_k$ is odd, ${}_A\mathfrak{V}_{\xi,k,\mathrm{sym}}$ is equal to
\begin{equation*}
\begin{cases}
  0&\text{when }k<Q,
  \\
  \mathfrak{U}_{E/E_{R+1},{\mathrm{sym-unram}}}=\bigoplus_{[\sigma^i\phi^{f/2}]\in(\Gamma_E\backslash \Gamma_{E_{R}}/\Gamma_E)_{\mathrm{sym}} }\mathfrak{U}_{[\sigma^i\phi^{f/2}]}&\text{when }k=Q,
  \\
  \mathfrak{U}_{E_k/E_{k+1},{\mathrm{sym-unram}}}=\bigoplus_{[\sigma^i\phi^{f/2}]\in(\Gamma_E\backslash (\Gamma_{E_{k+1}}-\Gamma_{E_{k}})/\Gamma_E)_{\mathrm{sym}} }\mathfrak{U}_{[\sigma^i\phi^{f/2}]}&\text{when }k>Q.
\end{cases}
\end{equation*}

We observe that, whether $a_k$ is odd or even, the symmetric unramified part of ${}_A\mathfrak{V}_{\xi}$ and ${}_M\mathfrak{V}_{\xi}$ are complementary, in the sense that
  $${}_A\mathfrak{V}_{\xi,k,{\mathrm{sym-unram}}}\oplus{}_M\mathfrak{V}_{\xi,k,{\mathrm{sym-unram}}}=\mathfrak{U}_{E_k/E_{k+1},{\mathrm{sym-unram}}}$$
for all $k=0,\dots,t$

We summarize the above in the following.
\begin{prop}\label{sym mod}
We always have $ {}_A\mathfrak{V}_{\xi,{\mathrm{sym-ram}}}\cong {}_M\mathfrak{V}_{\xi,{\mathrm{sym-ram}}}$ and
  \begin{enumerate}[(i)]
        \item when $f$ is odd, or when both $f$ and $m$ are even, then ${}_A\mathfrak{V}_{\xi,{\mathrm{sym-unram}}}\cong {}_M\mathfrak{V}_{\xi,{\mathrm{sym-unram}}}$;
            \item when $f$ is even and $m$ is odd, then $ {}_A\mathfrak{V}_{\xi,{\mathrm{sym-unram}}}\oplus{}_M\mathfrak{V}_{\xi,{\mathrm{sym-unram}}}=\mathfrak{U}_{{\mathrm{sym-unram}}}.$
  \end{enumerate}
\end{prop}

\subsection{t-factors of isotypic components}

We recall the values of the t-factors $t_{{\boldsymbol{\Gamma}}}^{i}(\mathfrak{V})$, $i=0,1$, when ${\boldsymbol{\Gamma}}$ is one of the cyclic subgroups ${\boldsymbol{\mu}}$ and $\varpi$ of $\Psi_{E/F}$ defined in (\ref{finite cyclic subgp}), and $\mathfrak{V}$ is a symmetric module $\boldsymbol{\mathfrak{U}}_{[g]}$ defined in (\ref{def sym mod}). The following Proposition describes all $t_{{\boldsymbol{\Gamma}}}^{i}(\boldsymbol{\mathfrak{U}}_{[g]})$ except when $[g]=[\sigma^{e/2}]$.

\begin{prop}[{\cite[Proposition 4.9]{thesis}}]\label{summary of t-factors}
\begin{enumerate}[(i)]
\item If $[g]=[\sigma^i\phi^j]$ is asymmetric, then
\begin{equation*}
  \begin{split}
    &t^0_{{\boldsymbol{\mu}}}(\boldsymbol{\mathfrak{U}}_{[g]})=1,\,\quad t^1_{{\boldsymbol{\mu}}}  (\boldsymbol{\mathfrak{U}}_{[g]}) :{{z}}\mapsto    \mathrm{sgn}_{{{z}}^{q^{i}-1}}(\mathfrak{U}_{[g]}),
    \\
    &t^0_{\varpi}(\boldsymbol{\mathfrak{U}}_{[g]})=1,\text{ and }t^1_{\varpi}(\boldsymbol{\mathfrak{U}}_{[g]})(\varpi_E) = \mathrm{sgn}_{{{z}}^i_e {{z}}_{\phi^j}}(\mathfrak{U}_{[g]}).
  \end{split}
\end{equation*}
    \item If $[g]=[\sigma^i]$ is symmetric and not equal to $[1]$ or $[\sigma^{e/2}]$, then
        \begin{equation*}
  \begin{split}
    &t^0_{{\boldsymbol{\mu}}}(\mathfrak{U}_{[g]})=1,\,\quad t^1_{{\boldsymbol{\mu}}}(\mathfrak{U}_{[g]}) \equiv1,
    \\
    & t^0_\varpi(\mathfrak{U}_{[g]})= -1, \text{ and }
    t^1_\varpi(\mathfrak{U}_{[g]}):\varpi_E\mapsto\left(\frac{{{z}}_e^i}{\ker(N_{\mathbb{F}_p[{{z}}^i_e]/\mathbb{F}_p[{{z}}^i_e]_\pm})}\right).
  \end{split}
\end{equation*}
    \item If $[g]=[\sigma^i\phi^{f/2}]$ is symmetric, then
    \begin{equation*}
       t^0_{{\boldsymbol{\mu}}}(\mathfrak{U}_{[g]})=-1,\,\quad t^1_{\boldsymbol{\mu}}(\mathfrak{U}_{[g]})\text{ is quadratic},
    \end{equation*}
and         \begin{enumerate}[(I)]
\item if ${{z}}_e^i{{z}}_{\phi^{f/2}}=1$, then $t^0_{\varpi}(\mathfrak{U}_{[g]}) = 1$ and $t^1_{\varpi}(\mathfrak{U}_{[g]}) \equiv 1$;
\item if ${{z}}_e^i{{z}}_{\phi^{f/2}}=-1$, then $t^0_{\varpi}(\mathfrak{U}_{[g]})=1$ and $t^1_{\varpi}(\mathfrak{U}_{[g]})(\varpi_E)=(-1)^{\frac{1}{2}(q^{f/2}-1)}$;
\item if ${{z}}_e^i{{z}}_{\phi^{f/2}}\neq\pm1$, then $t^0_{\varpi}(\mathfrak{U}_{[g]}) = -1$ and $$t^1_{\varpi}(\mathfrak{U}_{[g]}):\varpi_E\mapsto \left(\frac{{{z}}_e^i{{z}}_{\phi^{f/2}}}{\ker(N_{\mathbb{F}_p[{{z}}_e^i{{z}}_{\phi^{f/2}}]/\mathbb{F}_p[{{z}}_e^i{{z}}_{\phi^{f/2}}]_\pm})}\right).$$
\end{enumerate}
\end{enumerate}
\qed\end{prop}

In the exceptional case, when $[g]=[\sigma^{e/2}]$, we have ${\boldsymbol{\mu}}_{E_g}={\boldsymbol{\mu}}_E$. To unify notation, we define
\begin{equation}\label{define t1-mu-U-sigma{e/2}}
  t^1_{\boldsymbol{\mu}}(\mathfrak{U}_{[\sigma^{e/2}]}):{\boldsymbol{\mu}}_{E}\rightarrow\{\pm1\},\,x\mapsto \left(\frac{x}{{\boldsymbol{\mu}}_E}\right).
\end{equation}
The $\mathbb{F}_p\varpi$-module structure of $\mathfrak{U}_{[\sigma^{e/2}]}$ does not concern us (see the sentence after Formula (\ref{zeta-data-e/2})).

The following properties concerning symmetric double cosets are useful when computing the above t-factors.
\begin{prop}\label{order 2 trick}
Suppose that $[g]$ is symmetric.
\begin{enumerate}[(i)]
  \item If $[g]$ is ramified (resp. unramified), then $[\mathfrak{U}_{[g]}:\mathbf{k}_E]$ is even (resp. odd).
      \item Let $\mathbb{F}_p[\left[\begin{smallmatrix}
1\\  g
\end{smallmatrix}\right](\varpi_E)]$ be the field extension of $\mathbb{F}_p$ generated by the image of $\left[\begin{smallmatrix}
1\\  g
\end{smallmatrix}\right](\varpi_E)$ in $ \bar{\mathbf{k}}^\times_F$. If $[g]\neq [\sigma^{e/2}]$, then the degree $[\mathfrak{U}_{[g]}:\mathbb{F}_p[\left[\begin{smallmatrix}
1\\  g
\end{smallmatrix}\right](\varpi_E)]]$ is odd.
\end{enumerate}
\end{prop}
\proof
The first statement for ramified $[g]$ is a simple calculation, and that for unramified $[g]$ is a consequence of Proposition \ref{parity of double coset with i=f/2}. The second statement is proved in {\cite[Lemma 4.8]{thesis}}.
\qed

We would like to extend our definition of the t-factors $t^{i}_{{\boldsymbol{\mu}}}(\boldsymbol{\mathfrak{U}}_{[g]})$, with $i=0,1$, from ${\boldsymbol{\mu}}$ to ${\boldsymbol{\mu}}_{g}={\boldsymbol{\mu}}_{E_g}/{\boldsymbol{\mu}}_F$. We define
\begin{equation*}\label{definition extended t-factor-0}
t^{0}_{{\boldsymbol{\mu}}_{g}}(\boldsymbol{\mathfrak{U}}_{[g]}):=t^{0}_{{\boldsymbol{\mu}}}(\boldsymbol{\mathfrak{U}}_{[g]})
\end{equation*}
and for all ${{z}}\in {\boldsymbol{\mu}}_{g}$,
\begin{equation*}\label{definition extended t-factor-1}
t^{1}_{{\boldsymbol{\mu}}_{g}}(\boldsymbol{\mathfrak{U}}_{[g]}):{{z}}\mapsto
\begin{cases}
\text{sgn}_{\left[\begin{smallmatrix}
  1 \\ g
\end{smallmatrix}\right]({{z}})}(\mathfrak{U}_{[g]}) & \text{ if }[g]\text{ is asymmetric,}
  \\
\left(\frac{\left[\begin{smallmatrix}
  1 \\ \sigma^{i}\phi^{f[\mathfrak{U}_{[g]}:\mathbf{k}_E]/2}
\end{smallmatrix}\right]({{z}})}{\ker N_{\mathbf{k}_{E_g}/\mathbf{k}_{E_{\pm g}}}}\right)  & \text{ if }[g]\text{ is symmetric.}
\end{cases}
\end{equation*}

\begin{prop}
  The restriction $t^{1}_{{\boldsymbol{\mu}}_{g}}(\boldsymbol{\mathfrak{U}}_{[g]})$ to ${\boldsymbol{\mu}}$ is $t^{1}_{{\boldsymbol{\mu}}}(\boldsymbol{\mathfrak{U}}_{[g]})$ .
\end{prop}
\begin{proof}
  For asymmetric $[g]$, the result is immediate by definition. For symmetric ramified $[g]$, the restriction of the root $\left[\begin{smallmatrix}
  1 \\ \sigma^{i}\phi^{f[\mathfrak{U}_{[g]}:\mathbf{k}_E]}
\end{smallmatrix}\right]$ to ${\boldsymbol{\mu}}={\boldsymbol{\mu}}_E/{\boldsymbol{\mu}}_F$ is trivial, so the assertion is again true. When $[g]$ is symmetric unramified, we have to show that the restriction of  $\left(\frac{\cdot}{\ker N_{\mathbf{k}_{E_g}/\mathbf{k}_{E_{\pm g}}}}\right)$ to ${\boldsymbol{\mu}}$ is $\left(\frac{\cdot}{\ker N_{\mathbf{k}_{E}/\mathbf{k}_{E\pm}}}\right)$, or equivalently, to show that the index of the subgroup $\ker N_{\mathbf{k}_{E}/\mathbf{k}_{E\pm}}\cong{\boldsymbol{\mu}}_{q^{f/2}+1}$  of $\ker N_{\mathbf{k}_{E}/\mathbf{k}_{E\pm}}\cong{\boldsymbol{\mu}}_{q^{f[\mathfrak{U}_{[g]}:\mathbf{k}_E]/2}+1}$ is odd, which follows from Proposition \ref{order 2 trick}.
\end{proof}

\section{Zeta-data}

\subsection{Admissible embeddings of L-tori}\label{section Admissible embeddings of tori}

As mentioned in Section \ref{section-relation-previous}, to understand $\zeta$-data, it is better to first understand $\chi$-data, which is motivated by constructing admissible embeddings of L-tori \cite[Section 2.6]{LS}.

We take $T$ to be an elliptic torus of $G$ isomorphic to $\mathrm{Res}_{E/F}\mathbb{G}_m$. Its dual torus $\hat{T}$ is $\mathrm{Ind}_{E/F}(\mathbb{C}^\times)$, which is isomorphic to $(\mathbb{C}^\times)^n$ as a group. It is equipped with the induced action of the Weil group $\mathcal{W}_F$, which factors through the action of the Galois group $\Gamma_F$. We define the L-torus ${}^LT:=\hat{T}\rtimes \mathcal{W}_F$ as the L-group of $T$.

We assume that the dual torus $\hat{T}$ is embedded into the L-group ${}^LG=\hat{G}\times \mathcal{W}_F$ of $G$, where $\hat{G}=\mathrm{GL}_n(\mathbb{C})$, with image $\mathcal{T}$. For convenience, we simply denote the image of $t\in \hat{T}$ by the embedding  $\hat{T}\rightarrow\mathcal{T}\subset\hat{G}$ also by $t\in \mathcal{T} $. This embedding should be defined using the chosen splittings of $G$ and $\hat{G}$. As we do not need the full detail of the definition of this embedding, we only refer to \cite[Section 2.5]{LS} for details (or, when $(G,T)=(\mathrm{GL}_n,\mathrm{Res}_{E/F}\mathbb{G}_m)$, see \cite[Section 6.1]{thesis}). All we need to know is that we can always assume that the image $\mathcal{T}$ is the diagonal subgroup of $\hat{G}$.

With the embedding $\hat{T}\rightarrow\mathcal{T}$ chosen, an \emph{admissible embedding} from ${}^L T$ to ${}^L G$ is a morphism of groups $I: {}^LT \rightarrow {}^LG$ of the form
\begin{equation*}
I(t \rtimes w) = t {I}(1\rtimes w) \text{ for all }t\rtimes w\in {}^LT.
\end{equation*}
 Note that an admissible embedding maps $\mathcal{W}_F$ into $N_{\hat{G}}(\mathcal{T})$, i.e., the factor ${I}(1\rtimes w)$ above lies in $N_{\hat{G}}(\mathcal{T})$. Two admissible embeddings $I_1$, $I_2$ are called $\mathrm{Int}(\mathcal{T})$-equivalent if there is $t\in \mathcal{T}$ such that $${I_1}(w)=t{I}_2(w)t^{-1}\text{ for all }w\in \mathcal{W}_F.$$

By \cite[Section 2.6]{LS}, admissible embeddings exist, and the collection of these embeddings can be described as follows.
\begin{prop}\label{AE as torsor}
  The set of admissible embeddings from ${}^LT $ to ${}^ LG$ is a $Z^1(\mathcal{W}_F,\hat{T})$-torsor, and the set of the $\mathrm{Int}(\mathcal{T})$-equivalence classes of these embeddings is an $H^1(\mathcal{W}_F,\hat{T})$-torsor.
\end{prop}

The idea in \cite[Section 2.5]{LS} of constructing an admissible embedding is to choose a set of characters
$$\{ \chi_{\lambda} \}_{\lambda \in \Phi}\text{, where }\chi_{\lambda}:E^\times_{\lambda} \rightarrow \mathbb{C}^\times,$$
called \emph{$\chi$-data}, such that the following conditions hold.
\begin{defn}\label{chi-data condition}
  \begin{enumerate}[(i)]
\item For each $\lambda\in\Phi$, we have $ \chi_{-\lambda} = \chi_\lambda^{-1}$ and $\chi_{{}^w\lambda} = {}^w\chi_\lambda$ for all $w \in \mathcal{W}_F$. \label{chi-data condition transition}
\item If $\lambda$ is symmetric, then $\chi|_{E^\times_{\pm \lambda}}$ equals the quadratic character $\delta_{E_{\lambda}/E_{\pm \lambda}}$ attached to the extension ${E_{\lambda}/E_{\pm \lambda}}$. \label{chi-data condition quadratic}
\end{enumerate}
\end{defn}

Remember that, in Section \ref{section root system}, we choose a subset $\mathcal{R}_\pm=\mathcal{R}_\mathrm{sym}\sqcup \mathcal{R}_{\mathrm{asym}/\pm}$ of $\Phi$ representing the orbits
$ \mathcal{W}_F\backslash\Phi_\text{sym}$ and $ \mathcal{W}_F\backslash\Phi_{\mathrm{asym}/\pm}$. Hence, by condition (\ref{chi-data condition transition}), the set of $\chi$-data depends completely on the subset $\{ \chi_\lambda\}_{\lambda \in \mathcal{R}_\pm}$. We still call such a subset a set of $\chi$-data. Moreover, using Artin reciprocity \cite{Tate-NTB}, we may regard each $\chi_\lambda$ as a character of the Weil group $\mathcal{W}_{E_\lambda}$.

Following the recipe in \cite[Section 2.5]{LS}, we can define an admissible embedding
$$I_{\{{\chi_\lambda}\}}:{}^LT\rightarrow {}^LG$$
depending on a given set of $\chi$-data. In our present situation, we can describe the admissible embedding $I_{\{{\chi_\lambda}\}}$ in Proposition \ref{recover the character from induction} below. We first recall the Langlands correspondence for the torus $T=\mathrm{Res}_{E/F}\mathbb{G}_m$, which is a bijection
\begin{equation}\label{hom as H1}
\text{Hom}(T(F), \mathbb{C}^\times) \rightarrow H^1(\mathcal{W}_F, \hat{T}) .
\end{equation}
Given a character $\xi$ of $T(F)=E^\times$, we denote by $\tilde{\xi}$ a 1-cocycle in $Z^1(\mathcal{W}_F,\hat{T})$ whose class is the image of $\xi$ under (\ref{hom as H1}). Given $\chi$-data $\{\chi_\lambda\}_{\lambda\in\mathcal{R}_\pm}$, we define
\begin{equation*}\label{mu as a product of chi-data}
\mu :=\mu_{\{\chi_\lambda\}}=\prod_{\lambda\in \mathcal{R}}\mathrm{Res}^{E^\times_\lambda}_{E^\times}\chi_\lambda,
\end{equation*}
where $\mathcal{R}=\mathcal{R}_\pm\sqcup(-\mathcal{R}_{\mathrm{asym}/\pm})$ is a subset representing $\Gamma_F\backslash\Phi$. It is easy to check that the product of the restricted characters is independent of representatives in $ \mathcal{R}$, so we usually write
\begin{equation*}
\mu =\prod_{[\lambda]\in\Gamma_F\backslash\Phi}\mathrm{Res}^{E^\times_\lambda}_{E^\times}\chi_\lambda.
\end{equation*}

\begin{prop}[\text{\cite[Proposition 6.5]{thesis}}]\label{recover the character from induction}
For every character $\xi$ of $E^\times$, the composition
$$I_{\{\chi_\lambda\}}\circ\tilde{\xi}:\mathcal{W}_F \rightarrow {}^LT \rightarrow {}^LG \xrightarrow{\text{proj}}\mathrm{GL}_n(\mathbb{C})$$
 is isomorphic to $\mathrm{Ind}_{E/F}(\xi \cdot \mu_{\{\chi_\lambda\}})$ as a representation of $\mathcal{W}_F$.
\qed\end{prop}

We now define a analogous set of characters
$$\{ \zeta_{\lambda} \}_{\lambda \in \Phi}\text{, where }\zeta_{\lambda}:E^\times_{\lambda} \rightarrow \mathbb{C}^\times,$$
called \emph{$\zeta$-data}, such that the following conditions hold.
\begin{defn}\label{zeta-data condition}
  \begin{enumerate}[(i)]
\item For each $\lambda\in\Phi$, we have $ \zeta_{-\lambda} = \zeta_\lambda^{-1}$ and $\zeta_{{}^w\lambda} = {}^w\zeta_\lambda$ for all $w \in \mathcal{W}_F$. \label{zeta-data condition transition}
\item If $\lambda$ is symmetric, then $\zeta|_{E^\times_{\pm \lambda}}$ is trivial. \label{zeta-data condition quadratic}
\end{enumerate}
\end{defn}
We can view a set of $\zeta$-data as the difference of two sets of $\chi$-data. Motivated from Propositions \ref{AE as torsor} and  \ref{recover the character from induction}, the product character
\begin{equation*}\label{nu as a product of chi-data}
\nu :=\nu_{\{\chi_g\}}=\prod_{[g]\in( \mathcal{W}_E \backslash \mathcal{W}_F/\mathcal{W}_E)'}\mathrm{Res}^{E^\times_g}_{E^\times}\zeta_g.
\end{equation*}
can be viewed as measuring the difference of two admissible embeddings.

Recall that, similar to choosing $\mathcal{R}_\pm$, we can also choose $\mathcal{D}_\pm=\mathcal{D}_\mathrm{sym}\sqcup \mathcal{D}_{\mathrm{asym}/\pm}$ to be a subset of $\Gamma_F/\Gamma_E$ consisting of representatives of $(\Gamma_E\backslash \Gamma_F/\Gamma_E)_\mathrm{sym}$ and $(\Gamma_E\backslash \Gamma_F/\Gamma_E)_{\mathrm{asym}/\pm}$ respectively, and obtain a bijection from Proposition \ref{orbit of roots as double coset},
\begin{equation*}\label{bijection between representatives}
  \mathcal{R}_\pm=\mathcal{R}_\text{sym}\bigsqcup\mathcal{R}_{\mathrm{asym}/\pm}\rightarrow \mathcal{D}_\pm=\mathcal{D}_\text{sym}\bigsqcup\mathcal{D}_{\mathrm{asym}/\pm},\,\lambda=\left[\begin{smallmatrix}1\\g\end{smallmatrix}\right]\mapsto g.
\end{equation*}
We usually denote by $E_g$ and $E_{\pm g}$ the fields $E_ \lambda$ and $E_{\pm \lambda}$ respectively, if ${g\in\mathcal{D}_\pm}$ corresponds to ${\lambda\in\mathcal{R}_\pm}$. We also denote by $\zeta_g$ the character $\zeta_\lambda$, and write
\begin{equation*}
\nu :=\nu_{\{\chi_g\}}=\prod_{[g]\in( \mathcal{W}_E \backslash \mathcal{W}_F/\mathcal{W}_E)'}\mathrm{Res}^{E^\times_g}_{E^\times}\zeta_g.
\end{equation*}

\subsection{Symmetric unramified zeta-data }\label{section Symmetric unramified zeta-data}

We choose a specific $\zeta$-data $\zeta_g$ for each $[g]=[\sigma^i\phi^{f/2}]\in ( \mathcal{W}_E \backslash \mathcal{W}_F/\mathcal{W}_E)_{{\mathrm{sym-unram}}}$, base on the results from the $\chi$-datum $\chi_g$.

Notice that, since $E_{g}/E_{\pm g}$ is quadratic unramified, the norm group $N_{E_{g}/E_{\pm g}}({E^\times_{g}})$ has a decomposition
$${\boldsymbol{{\boldsymbol{\mu}}}}_{E_{\pm g}}\times\left<{{z}}_e^i{{z}}_{\phi^{f/2}}\varpi_E^2\right>\times U^1_{E_{\pm g}}$$
and we take a root of unity ${{z}}_0\in {\boldsymbol{{\boldsymbol{\mu}}}}_{E_{g}}$ such that
 $${{z}}_0\varpi_E\in E_{\pm g}^\times-N_{E_{g}/E_{\pm g}}({E^\times_{g}}).$$

 We only consider tamely ramified $\chi$-data and $\zeta$-data, i.e., we require that
 $$\chi_g|_{U^1_{E_{ g}}}\equiv 1\text{ and }\zeta_g|_{U^1_{E_{ g}}}\equiv 1.$$
 Therefore, the Definition \ref{chi-data condition}.(\ref{chi-data condition quadratic}) of $\chi$-data is explicitly (see \cite[(7.6)]{thesis})
 \begin{equation}\label{3 conditions to check chi in S-0 other k}
\chi_{g}({\boldsymbol{\mu}}_{E_\pm g}) = 1,\,\chi_{g}({{z}}_e^i{{z}}_{\phi^{f/2}}\varpi^2_E) = 1,\text{ and }\chi_{g}({{z}}_0\varpi_E) = -1.
\end{equation}
Hence, given a $\chi$-datum $\chi_g$, we can obtain a $\zeta$-datum $\zeta_g$ easily by requiring
$$\zeta_g|_{{\boldsymbol{\mu}}_{E_g}}=\chi_g|_{{\boldsymbol{\mu}}_{E_g}}\text{ and } \zeta_g(\varpi_E)=-\chi_g(\varpi_E).$$

In \cite[Section 7.4]{thesis}, in the cases when $\mathfrak{V}_{\xi,[g]}\cong \mathfrak{U}_{[g]}$ is non-trivial, we construct a $\chi$-datum
$$\chi_g|_{{\boldsymbol{\mu}}_{E_g}}=t^1_{{\boldsymbol{\mu}}_g}(\mathfrak{U}_{[g]})
\text{ and }\chi_g(\varpi_E)=
\begin{cases}
  -t_\varpi(\mathfrak{U}_{[g]}) &\text{if }{}^{\sigma^i\phi^{f/2}}\varpi_E=\varpi_E,
  \\
  t_\varpi(\mathfrak{U}_{[g]})&\text{otherwise}.
\end{cases}$$
In other words, the character $\chi_g$ satisfies the conditions in (\ref{3 conditions to check chi in S-0 other k}). Hence the character
\begin{equation*}\label{sym-unram zeta datum}
  \zeta_g|_{{\boldsymbol{\mu}}_{E_g}}=t^1_{{\boldsymbol{\mu}}_g}(\mathfrak{U}_{[g]})
\text{ and }\zeta_g(\varpi_E)=
\begin{cases}
  t_\varpi(\mathfrak{U}_{[g]}) &\text{if }{}^{\sigma^i\phi^{f/2}}\varpi_E=\varpi_E,
  \\
  -t_\varpi(\mathfrak{U}_{[g]})&\text{otherwise}
\end{cases}
\end{equation*}
is a $\zeta$-datum. This $\zeta$-datum will be used in the next section.

\subsection{Zeta-data associated to admissible characters}

Given an admissible character $\xi$ of $E^\times$ over $F$, we first assign, for each $[g]\in (\mathcal{W}_E\backslash \mathcal{W}_{F}/\mathcal{W}_E)_{\mathrm{asym}/\pm}$, the values of the $\zeta$-data to $$\zeta_{g,\xi}|_{{\boldsymbol{\mu}}_{E_{g}}}=\mathrm{sgn}_{{\boldsymbol{\mu}}_{E_{g}}}({}_A\mathfrak{V}_{[g]})\mathrm{sgn}_{{\boldsymbol{\mu}}_{E_{g}}}({}_M\mathfrak{V}_{[g]})=
t^1_{{\boldsymbol{\mu}}_{{g}}}({}_A\boldsymbol{\mathfrak{V}}_{\xi,[g]})t^1_{{\boldsymbol{\mu}}_{{g}}}({}_M\boldsymbol{\mathfrak{V}}_{\xi,[g]}).$$
In this way, the product of the characters
$$\zeta_{g,\xi}\zeta_{g^{-1},\xi}=\zeta_{g,\xi}\left(\zeta_{g,\xi}^{g}\right)^{-1}= \zeta_{g,\xi}\circ\left[\begin{smallmatrix}
1\\  g
\end{smallmatrix}\right]$$
has values
\begin{equation*}\label{value on mu-E for asymmetric}
\begin{split}
\left(\zeta_{g,\xi}\circ\left[\begin{smallmatrix}
1\\  g
\end{smallmatrix}\right]\right)|_{{\boldsymbol{\mu}}_E}({{z}})&=\mathrm{sgn}_{{\boldsymbol{\mu}}_{E}}({}_A\mathfrak{V}_{\xi,[g]})(
\left[\begin{smallmatrix}
1\\  g
\end{smallmatrix}\right]({{z}})
)\mathrm{sgn}_{{\boldsymbol{\mu}}_{E}}({}_M\mathfrak{V}_{\xi,[g]})(
\left[\begin{smallmatrix}
1\\  g
\end{smallmatrix}\right]({{z}})
)
\\
 &=t^1_{{\boldsymbol{\mu}}}({}_A\boldsymbol{\mathfrak{V}}_{\xi,[g]})t^1_{{\boldsymbol{\mu}}}({}_M\boldsymbol{\mathfrak{V}}_{\xi,[g]})
\end{split}
\end{equation*}
and
\begin{equation}\label{asymmetric varpi not depend}
\begin{split}
\left(\zeta_{g,\xi}\circ\left[\begin{smallmatrix}
1\\  g
\end{smallmatrix}\right]\right)(\varpi_E)& =\zeta_{g,\xi}|_{{\boldsymbol{\mu}}_E}\left(\left[\begin{smallmatrix}
1\\  g
\end{smallmatrix}\right](\varpi_E)\right)
\\
& = \text{sgn}_{\left[\begin{smallmatrix}
1\\  g
\end{smallmatrix}\right](\varpi_E)}({}_A\mathfrak{V}_{\xi,[g]})\text{sgn}_{\left[\begin{smallmatrix}
1\\  g
\end{smallmatrix}\right](\varpi_E)}({}_M\mathfrak{V}_{\xi,[g]})
\\
&=t^1_{\varpi}({}_A\boldsymbol{\mathfrak{V}}_{\xi,[g]})(\varpi_E)t^1_{\varpi}({}_M\boldsymbol{\mathfrak{V}}_{\xi,[g]})(\varpi_E)
\\
&= t_{\varpi}({}_A\boldsymbol{\mathfrak{V}}_{\xi,[g]})t_{\varpi}({}_M\boldsymbol{\mathfrak{V}}_{\xi,[g]}).
\end{split}
\end{equation}
We then assign, for each $[g]\in (\mathcal{W}_E\backslash \mathcal{W}_{F}/\mathcal{W}_E)_{\mathrm{asym}/\pm}$, arbitrary values to $\zeta_{g,\xi}(\varpi_E)$ and $\zeta_{g^{-1},\xi}(\varpi_E)$, as long as the product satisfies (\ref{asymmetric varpi not depend}). (This phenomenon is comparable to \cite[Lemma 3.3.A]{LS}, as explained in \cite[Remark 7.2]{thesis}.) It is routine to check that each $\zeta_{g,\xi}$ is a $\zeta$-datum. Indeed, this checking is exactly the same as that in the $\chi$-data case (see \cite[Section 7.2]{thesis}), since Definition \ref{chi-data condition}.(\ref{chi-data condition transition}) is the same as that of $\chi$-data.

We then assign values to the $\zeta$-data for each $[g]\in (\mathcal{W}_E\backslash \mathcal{W}_{F}/\mathcal{W}_E)_{\mathrm{sym}}$ case by case.

\subsubsection{Case when $f$ is odd}

Recall from (\ref{sym submod f odd}) that
\begin{equation}\label{sym-mod equal implies t-factors equal}
 t^i_{\boldsymbol{\Gamma}}({}_A\mathfrak{V}_{\xi,\mathrm{sym}})=t^i_{\boldsymbol{\Gamma}}({}_M\mathfrak{V}_{\xi,\mathrm{sym}})\text{, for }i=0,1\text{ and }\boldsymbol{\Gamma}={\boldsymbol{\mu}}, \varpi.
\end{equation}
We assign the $\zeta$-data to the following values. If $e$ is odd (so that $m$ is odd since $m$ divides $e$), we assign all $\zeta_{g,\xi}$ to be trivial. If $e$ is even, then we just take all $\zeta_{g,\xi}$, $[g]\neq [\sigma^{e/2}]$, to be trivial and
$$  \zeta_{\sigma^{e/2},\xi}|_{{\boldsymbol{\mu}}_E}\equiv 1\text{ and }   \zeta_{\sigma^{e/2},\xi}(\varpi_{E})=(-1)^m.$$
To show that $\zeta_{\sigma^{e/2},\xi}$ is a $\zeta$-datum, notice that since $N_{E/E_{\pm\sigma^{e/2}}}(\varpi_E)=-\varpi_{E_{\pm\sigma^{e/2}}}=-\varpi_E^2$, and since
$$\zeta_{\sigma^{e/2},\xi}(\varpi_E)^2=\zeta_{\sigma^{e/2},\xi}(-1)\zeta_{\sigma^{e/2},\xi}(-\varpi_E^2)=(1)(1)=1,$$
we can assign $\chi_{\sigma^{e/2},\xi}(\varpi_E)$ to either 1 or $-1$ to obtain a $\zeta$-datum. By (\ref{sym-mod equal implies t-factors equal}), we can rewrite our assigned $\zeta$-data as
\begin{equation}\label{zeta-data-e/2}
  \begin{split}
    &\zeta_{g,\xi}|_{{\boldsymbol{\mu}}_E}=t^1_{{\boldsymbol{\mu}}}({}_A\mathfrak{V}_{\xi,[g]})t^1_{{\boldsymbol{\mu}}}({}_M\mathfrak{V}_{\xi,[g]})
    \\
    &\zeta_{g,\xi}(\varpi_E)=\begin{cases}
    t_{\varpi}({}_A\mathfrak{V}_{\xi,[g]})t_{\varpi}({}_M\mathfrak{V}_{\xi,[g]}) &\text{ if }g\neq \sigma^{e/2},
    \\
  (-1)^m t_{\varpi}({}_A\mathfrak{V}_{\xi,[g]})t_{\varpi}({}_M\mathfrak{V}_{\xi,[g]})&\text{ if }g= \sigma^{e/2}.
\end{cases}
     \end{split}
\end{equation}
Note that $t_{\varpi}({}_A\mathfrak{V}_{\xi,[\sigma^{e/2}]})$ is not defined (see the paragraph containing Formula (\ref{define t1-mu-U-sigma{e/2}})). In fact, we just take $t_{\varpi}({}_A\mathfrak{V}_{\xi,[\sigma^{e/2}]})=t_{\varpi}({}_M\mathfrak{V}_{\xi,[\sigma^{e/2}]})=1$, since ${}_A\mathfrak{V}_{\xi,[\sigma^{e/2}]}\cong {}_M\mathfrak{V}_{\xi,[\sigma^{e/2}]}$ and it is shown in \cite[Proposition 5.3]{thesis} that ${}_M\mathfrak{V}_{\xi,[\sigma^{e/2}]}$ is always trivial.

The product of $\zeta$-data is equal to
\begin{equation*}\label{prod of zeta-data, f odd}
\begin{split}
  &\prod_{[g]\in (\mathcal{W}_E\backslash \mathcal{W}_F/\mathcal{W}_E)'}\zeta_{g,\xi}(\varpi_E)\equiv 1
  \\
    \text{ and }&\prod_{[g]\in (\mathcal{W}_E\backslash \mathcal{W}_F/\mathcal{W}_E)'}\zeta_{g,\xi}(\varpi_E)=
  \begin{cases}
  (-1)^m &\text{ if }e\text{ is even,}
  \\
  1 &\text{ if }e\text{ is odd,}
\end{cases}
\end{split}
  \end{equation*}
which is rewritten as
\begin{equation*}
\begin{split}
  &\prod_{[g]\in (\mathcal{W}_E\backslash \mathcal{W}_F/\mathcal{W}_E)'}\zeta_{g,\xi}|_{{\boldsymbol{\mu}}_E}\equiv t^1_{{\boldsymbol{\mu}}}({}_A\mathfrak{V}_{\xi})t^1_{{\boldsymbol{\mu}}}({}_M\mathfrak{V}_{\xi})
  \\
    \text{ and }&\prod_{[g]\in (\mathcal{W}_E\backslash \mathcal{W}_F/\mathcal{W}_E)'}\zeta_{g,\xi}(\varpi_E)=
  \begin{cases}
  (-1)^mt_{\varpi}({}_A\mathfrak{V}_{\xi})t_{\varpi}({}_M\mathfrak{V}_{\xi}) &\text{ if }e\text{ is even,}
  \\
  t_{\varpi}({}_A\mathfrak{V}_{\xi})t_{\varpi}({}_M\mathfrak{V}_{\xi}) &\text{ if }e\text{ is odd.}
\end{cases}
\end{split}
\end{equation*}
The product is equal to the rectifier given in (\ref{value of rectifier, sim alg}),
$${}_D\nu_\xi|_{{\boldsymbol{\mu}}_E}\equiv 1\text{ and }{}_D\nu_\xi(\varpi_E)=(-1)^{m(d-1)}=(-1)^{e-m},$$
  when $E/F$ is totally ramified.

\subsubsection{Case when $f$ is even}\label{subsection f even}

Let $K$ be the maximal unramified extension of $E/F$. If we define ${}_{D_K}\nu_\xi$ to be the ramified part of ${}_{D}\nu_\xi$, which is also the rectifier corresponding to the admissible pair $(E/K,\xi)$, then we have
$${}_{D_K}\nu_\xi=\prod_{[g]\in (\mathcal{W}_E \backslash \mathcal{W}_{K}/\mathcal{W}_E)'}\zeta_{g,\xi}|_{E^\times},$$
as when $f$ is odd, and in particular
$${}_{D_K}\nu_\xi(\varpi_E)=(-1)^{e-m_K},$$
where $m_K=\gcd(e,m)$. Therefore, our plan is to distribute the sign
\begin{equation}\label{sim-alg-sign-sym-unram}
  (-1)^{e-m_K}(-1)^{{n-m+f_\varpi-m_\varpi}}=
  \begin{cases}
    1 &\text{if }m\text{ is even,}
    \\
    (-1)^{e+f_\varpi+1}&\text{if }m\text{ is odd,}
  \end{cases}
\end{equation}
 to each $\zeta_{g,\xi}(\varpi_E)$, where $[g]$ is symmetric unramified, multiplying the product of t-factors $t_{\varpi}({}_A\mathfrak{V}_{\xi,[g]})t_{\varpi}({}_M\mathfrak{V}_{\xi,[g]})$. As before, we separate the cases according to the parity of $m$.

When $m$ is even, recall from Proposition \ref{sym mod} that we have either
\begin{equation*}
\text{both }{}_A\mathfrak{V}_{\xi,k,\mathrm{sym}}\text{ and }{}_M\mathfrak{V}_{\xi,k,\mathrm{sym}}\text{ are trivial,}
\end{equation*}
or
\begin{equation*}
\text{both }{}_A\mathfrak{V}_{\xi,k,\mathrm{sym}}\text{ and }{}_M\mathfrak{V}_{\xi,k,\mathrm{sym}}\text{ are isomorphic to } \mathfrak{U}_{k,\mathrm{sym}}.
\end{equation*}
 We assign the trivial $\zeta$-data for all  $[g]\in (\mathcal{W}_E\backslash \mathcal{W}_F/\mathcal{W}_E)_{{\mathrm{sym-unram}}}$, so that
\begin{equation*}
 \zeta_{g,\xi}|_{{\boldsymbol{\mu}}_E}=t^1_{{\boldsymbol{\mu}}}({}_A\mathfrak{V}_{\xi,[g]})t^1_{{\boldsymbol{\mu}}}({}_M\mathfrak{V}_{\xi,[g]})\text{ and }\zeta_{g,\xi}(\varpi_E)=
    t_{\varpi}({}_A\mathfrak{V}_{\xi,[g]})t_{\varpi}({}_M\mathfrak{V}_{\xi,[g]}).
  \end{equation*}
 The product of $\zeta_{g,\xi}(\varpi_E)$ is trivial, or we can write
 \begin{equation*}
 \begin{split}
 &\prod_{[g]\in (\mathcal{W}_E\backslash \mathcal{W}_F/\mathcal{W}_E)'}\zeta_{g,\xi}|_{{\boldsymbol{\mu}}_E}=t^1_{{\boldsymbol{\mu}}}({}_A\mathfrak{V}_{\xi})t^1_{{\boldsymbol{\mu}}}({}_M\mathfrak{V}_{\xi})
 \\
 \text{ and }
 &\prod_{[g]\in (\mathcal{W}_E\backslash \mathcal{W}_F/\mathcal{W}_E)'}\zeta_{g,\xi}(\varpi_E)=
    t_{\varpi}({}_A\mathfrak{V}_{\xi})t_{\varpi}({}_M\mathfrak{V}_{\xi}).
 \end{split}
  \end{equation*}
Note that in the second product, the sign without t-factors is equal to (\ref{sim-alg-sign-sym-unram}), which is just 1.

When $m$ is odd, we have
  $${}_A\mathfrak{V}_{\xi,{\mathrm{sym-unram}}}\oplus{}_M\mathfrak{V}_{\xi,{\mathrm{sym-unram}}}=\mathfrak{U}_{{\mathrm{sym-unram}}}.$$
  We then assign the $\zeta$-data to be
      \begin{equation*}
      \begin{split}
 & \zeta_{g,\xi}|_{{\boldsymbol{\mu}}_E}=t^1_{{\boldsymbol{\mu}}}({}_A\mathfrak{V}_{\xi,[g]})t^1_{{\boldsymbol{\mu}}}({}_M\mathfrak{V}_{\xi,[g]})=t^1_{{\boldsymbol{\mu}}}(\mathfrak{U}_{[g]})
 \\
 \text{ and }&\zeta_{g,\xi}(\varpi_E)=
    -t_{\varpi}({}_A\mathfrak{V}_{\xi,[g]})t_{\varpi}({}_M\mathfrak{V}_{\xi,[g]})=-t_{\varpi}(\mathfrak{V}_{[g]}).
      \end{split}
   \end{equation*}
     for all symmetric unramified $[g]$ except the one which stabilizes $\varpi_E$, in which we assign
\begin{equation*}
      \begin{split}
 & \zeta_{g,\xi}|_{{\boldsymbol{\mu}}_E}=t^1_{{\boldsymbol{\mu}}}({}_A\mathfrak{V}_{\xi,[g]})t^1_{{\boldsymbol{\mu}}}({}_M\mathfrak{V}_{\xi,[g]})=t^1_{{\boldsymbol{\mu}}}(\mathfrak{U}_{[g]})
 \\
 \text{ and }&\zeta_{g,\xi}(\varpi_E)=
    t_{\varpi}({}_A\mathfrak{V}_{\xi,[g]})t_{\varpi}({}_M\mathfrak{V}_{\xi,[g]})=t_{\varpi}(\mathfrak{V}_{[g]}).
      \end{split}
   \end{equation*}
     In Section \ref{section Symmetric unramified zeta-data}, we checked that the above characters give rise to $\zeta$-data. The product of $\zeta$ is hence
   \begin{equation*}
 \begin{split}
 &\prod_{[g]\in (\mathcal{W}_E\backslash \mathcal{W}_F/\mathcal{W}_E)'}\zeta_{g,\xi}|_{{\boldsymbol{\mu}}_E}=t^1_{{\boldsymbol{\mu}}}({}_A\mathfrak{V}_{\xi})t^1_{{\boldsymbol{\mu}}}({}_M\mathfrak{V}_{\xi})
 \\
 \text{ and }
 &\prod_{[g]\in (\mathcal{W}_E\backslash \mathcal{W}_F/\mathcal{W}_E)'}\zeta_{g,\xi}(\varpi_E)=
    (-1)^{e+f_\varpi+1}t_{\varpi}({}_A\mathfrak{V}_{\xi})t_{\varpi}({}_M\mathfrak{V}_{\xi}),
 \end{split}
  \end{equation*}
      by Proposition \ref{parity sym unram not fixing varpiE}. Again in the second product, the sign without t-factors is equal to (\ref{sim-alg-sign-sym-unram}).

\subsection{The main theorem}

To summarize, we verified the following theorem.
\begin{thm}\label{zeta-data factor of BH-rectifier}
Let $\xi$ be an admissible character of $E^\times$ over $F$.
 \begin{enumerate}[(i)]
\item  Let ${}_A\mathfrak{V}_\xi$ (resp. ${}_M\mathfrak{V}_\xi$) be the finite symplectic module defined by $\xi$ when $G(F)={\mathrm{GL}}_m(D)$ (resp. when $G^*(F)={\mathrm{GL}}_n(F)$). The following conditions define a collection of $\zeta$-data $\{\zeta_{g,\xi}\}_{g\in\mathcal{D}_{\mathrm{asym}/\pm}\sqcup \mathcal{D}_{\mathrm{sym}}}$.
 \begin{enumerate}
 \item All $\zeta_{g,\xi}$ are tamely ramified.
 \item If $g\in\mathcal{D}_{\mathrm{asym}/\pm}$, then
 \begin{equation*}
\zeta_{g,\xi}|_{{\boldsymbol{\mu}}_{E_{g}}}=t^1_{{\boldsymbol{\mu}}_{{g}}}({}_A\boldsymbol{\mathfrak{V}}_{\xi,[g]})t^1_{{\boldsymbol{\mu}}_{{g}}}({}_M\boldsymbol{\mathfrak{V}}_{\xi,[g]})
\end{equation*}
and
\begin{equation*}
\begin{split}
&\zeta_{g,\xi}(\varpi_E)\text{ can be any value satisfying}
\\
&\zeta_{g,\xi}(\varpi_E)\zeta_{g^{-1},\xi}(\varpi_E)= t_{\varpi}({}_A\boldsymbol{\mathfrak{V}}_{\xi,[g]})t_{\varpi}({}_M\boldsymbol{\mathfrak{V}}_{\xi,[g]}).
\end{split}
\end{equation*}
 \item If $g\in\mathcal{D}_{\mathrm{sym}}$, then \begin{equation*}
\zeta_{g,\xi}|_{{\boldsymbol{\mu}}_{E_g}}= t^1_{{\boldsymbol{\mu}}_g}({}_A\mathfrak{V}_{\xi,[g]})t^1_{{\boldsymbol{\mu}}_g}({}_M\mathfrak{V}_{\xi,[g]})
 \end{equation*}
  and
  \begin{equation*}
\zeta_{g,\xi}(\varpi_E)=\epsilon_g t_\varpi({}_A\mathfrak{V}_{\xi,[g]})t_\varpi({}_M\mathfrak{V}_{\xi,[g]}),
  \end{equation*}
  where $\epsilon_g$ is equal to $1$ if $g\in(\mathcal{D}_{{\mathrm{sym-ram}}}-\{\sigma^{e/2}\})\cup \mathcal{W}_{F[\varpi_E]}$ and is equal to $(-1)^m$ if $g\in(\mathcal{D}_{{\mathrm{sym-unram}}}- \mathcal{W}_{F[\varpi_E]})\cup\{\sigma^{e/2}\}$.
    \label{theorem sym case}
\end{enumerate}\label{theorem character}
\item Let ${}_D\nu_\xi$ be the rectifier of $\xi$ and  $\{\zeta_{g,\xi}\}_{g\in\mathcal{D}_\pm}$ be the $\zeta$-data in (\ref{theorem character}), then $${}_D\nu_\xi=\prod_{[g]\in (\mathcal{W}_E \backslash \mathcal{W}_{F}/\mathcal{W}_E)'}\zeta_{g,\xi}|_{E^\times}.$$
\end{enumerate}
\qed\end{thm}

\begin{rmk}
As long as the $F$-dimension of the division algebra $D$ is fixed, the rectifier ${}_D\nu\xi$ is independent of the Hasse-invariant $h=h(D)$ of $D$, as stated in \cite[Theorem C]{BH-JLC}. This is because the modules ${}_A\mathfrak{V}_\xi$, where $A=\mathrm{Mat}_n(D)$ and $D$ ranges over all division algebra with same $F$-dimension, are all isomorphic to each other. Similarly, the $\zeta$-data $\{\zeta_{g,\xi}\}$ are independent of $h(D)$.
  \qed\end{rmk}

\subsection{Functorial property}

Let $K$ be an intermediate subfield in $E/F$, and write
$$n_K=n(E/K)=f_Ke_K=f(E/K)e(E/K)\text{ and }m_K=\gcd(m,n_K).$$
 Similar to Section \ref{section The correspondences}, we have the Jacquet-Langlands correspondence
$$JL_K:\mathcal{A}^\mathrm{et}_{n_K}(K)\rightarrow\mathcal{A}^\mathrm{et}_{m_K}(D_K),$$
between essentially tame supercuspidal representations of $G(F)_K={\mathrm{GL}}_{m_K}(D_K)$ and its split inner form $G^*(F)_K={\mathrm{GL}}_{n_K}(K)$. We can parametrize both collections by the admissible pairs in $P_{n_K}(K)$, and obtain the rectifier map
\begin{equation*}
{}_{D_K}\nu: P_{n_K}(K) \xrightarrow{{}_K\Pi} \mathcal{A}^\mathrm{et}_{n_K}(K) \xrightarrow{JL_K} \mathcal{A}^\mathrm{et}_{m_K}(D_K)\xrightarrow{{}_{D_K}\Pi^{-1}} P_{n_K}(K),
\end{equation*}
such that
\begin{equation*}
{}_{D_K}\nu(E/K, \xi) = (E/F, \xi\cdot {}_{D_K}\nu_\xi).
\end{equation*}
for a tamely ramified character ${}_{D_K}\nu_\xi$ of $E^\times$ for each pair $(E/K, \xi)\in P_{n_K}(K)$.

 With the embedding condition for $(E_0K,\mathfrak{A}_K)$ as discussed in Section \ref{subsection An embedding condition}, we define the subgroups (see also \cite[3.2 Proposition]{BH-JLC})
\begin{equation*}
  \begin{split}
    H^1_K=H^1(\Xi,\mathfrak{A})\cap G(F)_K\text{ and }J^1_K=J^1(\Xi,\mathfrak{A})\cap G(F)_K
  \end{split}
\end{equation*}
Each subgroup above admits a similar factorization as in (\ref{group H(xi) J(xi)}). We then obtain
 \begin{equation*}
  {}_{A_K}\mathfrak{V}_{\xi}=J^1_K/H^1_K={}_{A}\mathfrak{V}_{\xi}\cap \mathfrak{U}_{E/K}
\end{equation*}
and similarly for ${}_{M_K}\mathfrak{V}_{\xi}$.

Denote $\Psi_{E/K}=E^\times/K^\times U_E^1$, and view ${}_{A_K}\mathfrak{V}_{\xi}$ and ${}_{M_K}\mathfrak{V}_{\xi}$
 as $\mathbf{k}_K\Psi_{E/K}$-submodules of $\mathfrak{U}_{E/K}$. Denote the subgroups of $\Psi_{E/K}$ by
$${\boldsymbol{\mu}}_{E/K}={\boldsymbol{\mu}}_E/{\boldsymbol{\mu}}_K\text{ and }\varpi_{E/K}=\text{the subgroup generated by the image of }\varpi_E.$$
Using the results in Section \ref{section Values of rectifiers}, with the base field changed from $F$ to $K$, the values of ${}_{D_K}\nu_\xi$ is given by
\begin{equation*}
  \begin{split}
    &{}_{D_K}\nu_\xi|_{{\boldsymbol{\mu}}_E}=t^1_{{\boldsymbol{\mu}}_{E/K}}({}_{A_K}\mathfrak{V}_{\xi})t^1_{{\boldsymbol{\mu}}_{E/K}}({}_{M_K}\mathfrak{V}_{\xi})
    \\
    \text{ and }&{}_{D_K}\nu_\xi(\varpi_E)=(-1)^{n_K-m_K+f_{{\varpi,K}}-m_{{\varpi,K}}}t_{\varpi_{E/K}}({}_{A_K}\mathfrak{V}_{\xi})t_{\varpi_{E/K}}({}_{M_K}\mathfrak{V}_{\xi})
     \end{split}
\end{equation*}
for a prime element $\varpi_E\in E_0K$ (see the beginning of Section \ref{section Values of rectifiers}). Here
$$f_{{\varpi,K}}=f(E/K[\varpi_E])\text{ and }m_{{\varpi,K}}=\gcd(m_K,f_{{\varpi,K}})=\gcd(m,n_K,f(E/K[\varpi_E])).$$

Now suppose that $(E/F,\xi)\in P_n(F)$. By the definition of admissibility, we can regard $\xi$ as an admissible character over $K$ and form the pair $(E/K,\xi)\in P_{n_K}(K)$.
\begin{prop}\label{rectifier decomp over K}
  In this situation, we have
  $${}_{D_K}\nu_\xi=\prod_{[g]\in (\mathcal{W}_E \backslash \mathcal{W}_{K}/\mathcal{W}_E)'}\zeta_{g,\xi}|_{E^\times}.$$
\end{prop}
\proof
Notice that if $\mathfrak{V}$ is a $\mathbf{k}_F\Psi_{E/F}$-submodule of $\mathfrak{U}_{E/K}$, we have
$$t^1_{{\boldsymbol{\mu}}_{E/K}}(\mathfrak{V})=t^1_{{\boldsymbol{\mu}}_{E/F}}(\mathfrak{V})\text{ and }t_{\varpi_{E/K}}(\mathfrak{V})=t_{\varpi_{E/F}}(\mathfrak{V}),$$
where ${\boldsymbol{\mu}}_{E/F}$ and $\varpi_{E/F}$ are just ${\boldsymbol{\mu}}$ and $\varpi$ respectively considered in (\ref{finite cyclic subgp}). Hence we have
$${}_{D_K}\nu_\xi|_{{\boldsymbol{\mu}}_E}=\prod_{[g]\in (\mathcal{W}_E \backslash \mathcal{W}_{K}/\mathcal{W}_E)'}\zeta_{g,\xi}|_{{\boldsymbol{\mu}}_E}.$$
It remains to consider the values of both characters at $\varpi_E$. Notice that the left side has value
$$(-1)^{n_K-m_K+f_{{\varpi,K}}-m_{{\varpi,K}}}t_{\varpi_{E/K}}({}_{A_K}\mathfrak{V}_{\xi})t_{\varpi_{E/K}}({}_{M_K}\mathfrak{V}_{\xi}),$$
while the right side has value
$$(\text{a sign})\cdot \prod_{[g]\in (\mathcal{W}_E \backslash \mathcal{W}_{K}/\mathcal{W}_E)'}t_{\varpi}({}_{A_K}\mathfrak{V}_{\xi,[g]})t_{\varpi}({}_{M_K}\mathfrak{V}_{\xi,[g]}).$$
The t-factors on both sides are clearly equal. We will recall, by Theorem \ref{zeta-data factor of BH-rectifier}, the values of the sign on the right side in different cases and show that, in each case, this sign is equal to the one on the left side.
\\\\
We first consider when $f_K$ is odd, which can be reduced to the case when $E/K$ is totally ramified. We further separate into cases.
\begin{itemize}
  \item When $e$ is odd, or when $e$ is even and $e_K$ is odd, then $m_K$ is also odd. The sign on the left is $(-1)^{e_K-m_K}=1$, while that on the right is also 1 since $\sigma^{e/2}\notin \mathcal{W}_K$.
  \item When $e_K$ is even, the sign on the left is $(-1)^{e_K-m_K}=(-1)^m$ since $m_K\equiv m\mod 2$, while that on the right is $(-1)^m$ since $\sigma^{e/2}\in \mathcal{W}_K$.
\end{itemize}
We then consider then $f_K$ is even. Let $M$ be the maximal unramified extension of $E/K$. We recall, after disregarding the symmetric ramified component (as we did at the beginning of Sub-section \ref{subsection f even}), the sign on the left is equal to (see (\ref{sim-alg-sign-sym-unram}))
\begin{equation*}
  (-1)^{e_K-m_M}(-1)^{{n_K-m_K+f_{{\varpi,K}}-m_{{\varpi,K}}}}=
  \begin{cases}
    1 &\text{if }m_K\text{ is even,}
    \\
    (-1)^{e_K+f_{{\varpi,K}}+1}&\text{if }m_K\text{ is odd.}
  \end{cases}
\end{equation*}
Recall from Proposition \ref{parity sym unram not fixing varpiE} that the number $e_K+f_{{\varpi,K}}+1$ is just the cardinality of
$$(\Gamma_E\backslash \Gamma_K /\Gamma_E)_{\mathrm{sym-unram}}- \Gamma_{K[\varpi_E]}.$$
 Hence by Theorem \ref{zeta-data factor of BH-rectifier}, the sign on the right side is
$$(-1)^{m(e_K+f_{{\varpi,K}}+1)}.$$
By knowing that $m_K\equiv m\mod 2$, the sign above is equal to the one on the left side.
\qed

\providecommand{\bysame}{\leavevmode\hbox to3em{\hrulefill}\thinspace}
\providecommand{\MR}{\relax\ifhmode\unskip\space\fi MR }
\providecommand{\MRhref}[2]{%
  \href{http://www.ams.org/mathscinet-getitem?mr=#1}{#2}
}
\providecommand{\href}[2]{#2}

\end{document}